\documentclass{amsart}

\usepackage[T1]{fontenc}
\usepackage[utf8]{inputenc}

\usepackage{cite}
\usepackage{palatino}
\usepackage{hyperref}
\usepackage{amssymb,amsthm,units,mathabx}
\usepackage{color}
\usepackage[usenames,dvipsnames,svgnames,table]{xcolor}
\usepackage{mathtools,xparse,xspace}
\usepackage{rotating}
\usepackage{bm} 

\hypersetup{allbordercolors=red,linkcolor=red,urlcolor=blue,urlbordercolor=blue,colorlinks=false,pdfborder={0 0 1}} 

\usepackage{epsfig}  		
\usepackage{epic,eepic}       
\usepackage{tikz,pgf}
\usetikzlibrary{arrows}
\usepackage[abs]{overpic}

\DeclarePairedDelimiter{\floor}{\lfloor}{\rfloor}

\newtheorem{thm}{Theorem}[section]
\newtheorem{prop}[thm]{Proposition}
\newtheorem{lem}[thm]{Lemma}
\newtheorem{cor}[thm]{Corollary}
\newtheorem{conj}[thm]{Conjecture}

\newtheorem{const}[thm]{Construction}
\newtheorem{question}[thm]{Question}

\theoremstyle{definition}
\newtheorem{definition}[thm]{Definition}
\newtheorem{example}[thm]{Example}

\theoremstyle{remark}

\newtheorem{remark}[thm]{Remark}

\numberwithin{equation}{section}

\newcommand{\R}{\mathbb{R}}  
\newcommand{\Z}{\mathbb{Z}}  
\newcommand{\Q}{\mathbb{Q}}  
\newcommand{\bd}{\partial}  

\newcommand{\Mxi}{(M, \xi)} 
\newcommand{\Sigphi}{(\Sigma,\phi)} 
\newcommand{\PD}{\mathrm{PD}\,} 
\newcommand{\Spinc}{\mathrm{Spin}^{\textit{c}}}

\newcommand{\vect}[2]{\left(\begin{matrix} #1 \\ #2 \end{matrix}\right)} 
\newcommand{\vects}[2]{\left(\begin{smallmatrix} #1 \\ #2 \end{smallmatrix}\right)} 
\newcommand{\matrixb}[4]{{\left(\begin{matrix}#1 & #2 \\ #3 & #4\end{matrix} \right)}} 

\newcommand{\cf}{\widehat{CF}}
\newcommand{\hf}{\widehat{HF}}

\begin{document}

\begin{abstract}
We study the effect of surgery on transverse knots in contact 3-manifolds.  In particular, we investigate the effect of such surgery on open books, the Heegaard Floer contact invariant, and tightness.  The overarching theme of this paper is to show that in many contexts, surgery on transverse knots is more natural than surgery on Legendrian knots.

Besides reinterpreting surgery on Legendrian knots in terms of transverse knots, our main results on are in two complementary directions: conditions under which inadmissible transverse surgery (\textit{cf.\@} positive contact surgery on Legendrian knots) preserves tightness, and conditions under which it creates overtwistedness.  In the first direction, we give the first result on the tightness of inadmissible transverse surgery for contact manifolds with vanishing Heegaard Floer contact invariant.  In particular, inadmissible transverse surgery on the connected binding of a genus $g$ open book that supports a tight contact structure preserves tightness if the surgery coefficient is greater than $2g-1$.  In the second direction, along with more general statements, we deduce a partial generalisation to a result of Lisca and Stipsicz: when $L$ is a Legendrian knot with $tb(L) \leq -2$, and $|rot(L)| \geq 2g(L)+tb(L)$, then contact $(+1)$-surgery on $L$ is overtwisted.
\end{abstract}

\title{Transverse Surgery on Knots in Contact 3-Manifolds}
\author{James Conway}
\address{University of California, Berkeley}
\email{conway@berkeley.edu}
\maketitle

\section{Introduction}

On contact $3$-manifolds, surgery on Legendrian knots is a well-storied affair, but much less well-studied is surgery on transverse knots.  In this paper, we will systematically study transverse surgery: its relation to surgery on Legendrian knots; an algorithm for modifying an open book to describe the result of transverse surgery on a binding component; tightness results for inadmissible transverse surgeries on certain fibred knots; and cases where inadmissible transverse surgery results in an overtwisted manifold.  Along the way, we will see that many existing results regarding surgery on Legendrian knots are more naturally expressed in the language of transverse surgery.

\subsection{Admissible and Inadmissible Transverse Surgery}

Ding and Geiges \cite{DG:surgery} show that any contact manifold can be obtained by contact surgery on a Legendrian link in $S^3$ with its standard contact structure.  Our first result translates surgery on a Legendrian link into surgery on a transverse link.  Thus, combined with a previous result of Baldwin and Etnyre \cite{BE:transverse}, we can say the following:

\begin{thm}\label{thm:transverse surgery gets everything} Every contact 3-manifold $\Mxi$ can be obtained by transverse surgery on some oriented link in $(S^3,\xi_{\rm{std}})$.  In particular, a single contact $(-1)-$surgery (resp.\ $(+1)-$surgery) on a Legendrian knot corresponds to admissible (resp.\ inadmissible) transverse surgery on a transverse knot. \end{thm}

When translating back from transverse surgery to surgery on Legendrian knots, the situation is not as nice.  Baldwin and Etnyre \cite{BE:transverse} showed that for a large range of slopes, admissible transverse surgery corresponds to negative contact surgery on a Legendrian link in a neighbourhood of the transverse knot, but that for certain slopes, admissible transverse surgery was not comparable to any surgery on a Legendrian link.  For inadmissible surgeries, we show:

\begin{thm}\label{Transverse=Legendrian} Every inadmissible transverse surgery on a transverse knot corresponds to a positive contact surgery on a Legendrian approximation. \end{thm}

For contact $r$-surgeries on oriented Legendrian knots where the surgery coefficient $r$ has numerator and denominator larger than 1, there is more than one choice of contact structure for the surgery.  These correspond to ways to extend the contact structure over the surgery torus, and can be pinned down by choosing signs on bypass layers used to construct the surgery torus.  We show that inadmissible transverse surgery corresponds to the contact structure $\xi^-_r$ obtained by choosing all negative bypass layers.  Because of this, results that have been proved with some difficulty in the Legendrian setting are shown to be more naturally results about transverse knots and transverse surgery.  As an example of this, we reprove a result of Lisca and Stipsicz \cite{LS:surgery} (in the integral case) and Golla \cite{Golla} (in the rational case).

\begin{cor}\label{Legendrian Surgery on Stabilisation} Let $r > 0$ be a rational number.  For any Legendrian knot $L$, if $L_-$ is a negative stabilisation of $L$, then $\xi_r^-(L)$ is isotopic to $\xi_{r+1}^-(L_-)$. \end{cor}

\subsection{Open Book Decompositions}

Given any 3-manifold, there is an open book decomposition $\Sigphi$ describing it. Here $\Sigma$ is a surface with boundary and $\phi$ is an orientation preserving diffeomorphism of $\Sigma$ fixing its boundary.  The manifold is reconstructed from $\Sigma$ and $\phi$ as follows.  The mapping torus of $\Sigma$ with monodromy $\phi$ is a 3-manifold with torus boundary components; a closed manifold is now obtained by gluing in solid tori to the boundary components such that such that $\{* \in \bd \Sigma\} \times S^1$ bounds a disc and a longitude (the \textit{binding}) is mapped to the boundary components of the surfaces (the \textit{pages}).  We can define a contact structure on an open book decomposition that respects the decomposition, turning the binding into a transverse link.  Open books support a unique contact structure up to isotopy, and thanks to work of Giroux \cite{Giroux:OBD}, contact structures are supported by a unique open book up to a stabilisation operation.

Baker, Etnyre, and van Horn-Morris \cite{BEVHM} demonstrated that rational open books, \textit{ie.\@} those push-offs of the boundary of the pages are non-integral cables of the binding components (or more than one boundary component of a page meet the same binding component), also support a unique contact structure up to isotopy.  The curve that a page traces out on the boundary of a neighbourhood of the binding is calling the \textit{page slope}. Baker, Etnyre, and van Horn-Morris demonstrated that topological surgery on a binding component induces a rational open book, and if the surgery coefficient is less than the page slope, then the induced open book supports the contact structure obtained by admissible transverse surgery on that binding component.  We look at surgery on a binding component with coefficient greater than the page slope and show the following.

\begin{thm}\label{OBDThm} The open book induced by surgery on a binding component with coefficient greater than the page slope supports the contact structure coming from inadmissible transverse surgery on the binding component. \end{thm}

Compare this with Hedden and Plamenevskaya \cite{HP}, who discuss the same construction in the context of Heegaard Floer invariants.  They track properties of the contact structure induced by surgery on the binding of an open book.  We can now identify the contact structure they discuss as the one induced by inadmissible transverse surgery.

\begin{const}\label{OBD construction} We construct explicit integral open books that support admissible transverse surgery (less than the page slope) and inadmissible transverse surgery (with any surgery coefficient) on a binding component of an open book. \end{const}

\subsection{Tight Surgeries}

The Heegaard Floer package provided by Ozsváth and Szabó \cite{OS:hf1, OS:hf2} gives very powerful 3-manifold invariants, which are still being mined for new information.  Ozsváth and Szabó have shown \cite{OS:contact} how a fibred knot gives rise to an invariant $c(\xi)$ of the contact structure supported by the open book induced by the fibration.  They proved that the non-vanishing of this invariant implies the contact structure is tight.  Hedden and Plamenevskaya \cite{HP} showed that the same set-up and invariant exists when the knot is \textit{rationally fibred}, \textit{ie.\@} it is rationally null-homologous and fibred.

In light of Theorem~\ref{OBDThm}, Hedden and Plamenevskaya proved that if $K$ is a fibred knot in $\Mxi$ supporting $\xi$, and the contact class of $\xi$ is non-vanishing, then inadmissible transverse $r$-surgery, for $r \geq 2g$, where $g$ is the genus of $K$, preserves the non-vanishing of the contact invariant.  We extend this in two ways: first, we increase the range of the surgery coefficient to allow $r > 2g-1$, and second, we replace the non-vanishing of the Heegaard Floer contact invariant with the weaker condition of tightness.

\begin{thm}\label{thm:preserve tightness} If $K$ is an integrally fibred transverse knot in $\Mxi$ supporting $\xi$, where $\xi$ is tight (resp.\ has non-vanishing Heegaard Floer contact invariant), then inadmissible transverse $r$-surgery for $r > 2g-1$ results in a tight contact manifold (resp.\ a contact manifold with non-vanishing Heegaard Floer contact invariant).
\end{thm}

Given a knot $K \subset M$, consider $M \times [0,1]$, where $K \subset M \times \{1\}$.  Let $g_4(K)$ be the minimum genus of a surface in $M \times [0,1]$ with boundary $K$.  When $K \subset \Mxi$ is a non-fibred transverse knot, and there is a Legendrian approximation $L$ of $K$ with $tb(L) = 2g-1$, then we can conclude the following theorem.  Since its proof is very similar to that of Theorem~\ref{thm:preserve tightness}, we omit it; note that this result can also be extracted from the proof of a similar result of Lisca and Stipsicz for knots in $(S^3,\xi_{\rm{std}})$ \cite{LS:hf1}.

\begin{thm}\label{thm:non-fibred} If $K$ is a null-homologous transverse knot in $\Mxi$ with $g_4(K) > 0$, where $\xi$ has non-vanishing contact class, and there exists a Legendrian approximation $L$ of $K$ with $tb(L) = 2g_4(K)-1$, then inadmissible transverse $r$-surgery on $K$ preserves the non-vanishing of the contact class for $r > 2g_4(K)-1$. \end{thm}

Mark and Tosun \cite{MT} determine exactly when inadmissible transverse surgery on knots in $(S^3, \xi_{\rm{std}})$ preserves the non-vanishing of the contact class.  Their condition includes $sl(K) = 2\tau(K)-1$, where $\tau(K) \leq g_4(K)$ is the Heegaard Floer tau invariant, and allows for surgeries $r > 2 \tau(K) - 1$.  We do not recover their results entirely, as they do not require that there exists a Legendrian approximation with $tb = 2g-1$. However, we also discuss knots outside of $S^3$, which the result of Mark and Tosun (and results of Lisca and Stipsicz \cite{LS:surgery} and Golla \cite{Golla} that they generalise) does not.  Hedden and Plamenevskaya's \cite{HP} result does discuss knots outside of $S^3$, and Theorem~\ref{thm:preserve tightness} is a generalisation of their result.  In particular, Theorem~\ref{thm:preserve tightness} does not require non-vanishing of the contact invariant.

One might hope for an analogue of Theorem~\ref{thm:preserve tightness} for links, \textit{ie.\@} that given an open book with multiple boundary components supporting a tight contact structure, sufficiently large inadmissible transverse surgery on every binding component might yield a tight manifold.  We construct examples to show that no such theorem can exist.  In particular, we construct the following examples.

\begin{const}\label{link not tight} For every $g \geq 0$ and $n \geq 2$, there is an open book of genus $g$ with binding a link with $n$ components supporting a Stein fillable contact structure, where anytime inadmissible transverse surgery is performed on all the binding components, the result is overtwisted. \end{const}

Seeing that tightness and non-vanishing of the Heegaard Floer contact invariant is preserved under inadmissible transverse surgery for large enough surgery coefficients, it is natural to ask whether other properties are preserved.  Other properties that are preserved under contact $(-1)$-surgery are those of the various types of fillability: Stein, strong, or weak fillability.  Results of Eliashberg \cite{Eliashberg:stein}, Weinstein \cite{Weinstein}, and Etnyre and Honda \cite{EH:non-fillable} show that contact $(-1)$-surgery preserves fillability in each of these categories.

Given a contact manifold $\Mxi$, and a cover $M' \to M$, there is an induced contact structure $\xi'$ on $M'$.  We say $\Mxi$ is \textit{universally tight} if $\xi$ is tight and the induced contact structure on every cover is also tight; otherwise, we call the contact structure \textit{virtually overtwisted}.  Since the fundamental group of a compact $3$-manifold is residually finite (by geometrization), universal tightness is equivalent to requiring that all finite covers of $\Mxi$ remain tight \cite{Honda:classification1}.

For large surgery coefficients, inadmissible transverse surgery adds a very small amount of twisting to the contact manifold.  One might therefore expect that a sufficiently large surgery coefficient might preserve properties of the original contact structure, as in Theorem~\ref{thm:preserve tightness} and several other prior results.  However, using classification results of Honda \cite{Honda:classification2} and calculations of Lisca and Stipsicz \cite{LS:non-fillable}, we construct examples of open books of all genera that show that inadmissible transverse surgery does not in general preserve the property of being universally tight, nor any of the fillability categories, even with arbitrarily large surgery coefficient.

\begin{const}\label{non-fillable} For every $g$, there is a transverse knot of genus $g$ in a Stein fillable universally tight contact manifold, where sufficiently large inadmissible transverse surgery preserves the non-vanishing of the contact class, but where no inadmissible transverse surgery is universally tight or weakly semi-fillable. \end{const}

\subsection{Overtwisted Surgeries}

Lisca and Stipsicz \cite{LS:hfnotes} show that given a Legendrian $L \subset (S^3,\xi_{\rm{std}})$ with $tb(L) \leq -2$, contact $(+1)$-surgery on $L$ has vanishing Heegaard Floer contact invariant.  It is natural to ask whether these are overtwisted.  We answer this question for a large class of knots, not just in $S^3$, leaving out only a finite set of $(tb,rot)$ pairs for each knot genus $g$.  A sample result of this type is as follows.

\begin{thm}\label{overtwistedsurgery} Let $L$ be a null-homologous Legendrian in $\Mxi$, where $c_1(\xi)$ is torsion, with $tb(L) \leq -2$ and $|rot(L)| \geq 2g(L)+tb(L)$, and let $T$ be a positive transverse push-off of $L$.  Then inadmissible transverse $(tb(L)+1)$-surgery on $T$ is overtwisted.  If $rot(L) > 2g(L)+1+tb(L)$ (that is, if $sl(T) < -2g(L) - 1$), then all inadmissible transverse surgeries on $T$ are overtwisted. \end{thm}

Note that inadmissible transverse $(tb(L)+1)$-surgery on $T$ is equivalent to contact $(+1)$-surgery on $L$.  Using this and its generalisation to contact $(+n)$-surgery, we derive the following corollary.

\begin{cor}
For every genus $g$ and every positive integer $n\geq 2$, there is a negative integer $t$ such that if $L$ is a null-homologous Legendrian knot of genus $g$ and $tb(L) \leq t$, then contact $(+n)$-surgery on $L$ is overtwisted.
\end{cor}

Lisca and Stipsicz \cite{LS:hfnotes} show that all contact $(+1)$-surgeries on negative torus knots are overtwisted.  Our results allow for the following generalisation.

\begin{cor}
All inadmissible transverse surgeries on negative torus knots in $(S^3,\xi_{\rm{std}})$ are overtwisted.
\end{cor}

Note that using the above results, we cannot show that all inadmissible transverse surgeries on the Figure Eight knot in $(S^3,\xi_{\rm{std}})$ are overtwisted.  However, we can obtain this result using convex surface theory methods, see \cite{conway:f8}.

\subsection{Knot Invariants}

Given a knot type $K$ in a contact manifold $\Mxi$, we can consider all regular neighbourhoods of transverse representatives of $K$.  We define the \textit{contact width} of $K$ be the supremum of the slopes of the characteristic foliation on the boundary of these neighbourhoods.  The following follows from Theorem~\ref{thm:preserve tightness}.  Note that the contact width as first defined in \cite{EH:cables} is the reciprocal of our invariant, as is their convention for the slope of characteristic foliations.

\begin{cor}\label{cor:contact width}
Let $K$ be an integrally fibred transverse knot in a tight contact manifold $\Mxi$ such that the fibration supports the contact structure $\xi$.  If the maximum Thurston--Bennequin number of a Legendrian approximation of $K$ is $2g-1$, then $w(K)=2g-1$.
\end{cor}

We also extend an invariant of transverse knots defined by Baldwin and Etnyre \cite{BE:transverse}, and determine its value in certain cases.

\subsection{Organisation of Paper}

Section~\ref{sec:LegendrianTransverseSurgery} describes the surgery operation on and transverse and Legendrian knots, and the relation between them.  Section~\ref{sec:OBD} describes open book decompositions, and includes the construction of open books supporting transverse surgery on a binding component. Section~\ref{sec:tight surgeries} discusses Heegaard Floer homology, and contains a proof of Theorem~\ref*{thm:preserve tightness}.  Section~\ref{sec:fillability} discusses fillability, universal tightness, and Construction~\ref*{non-fillable}.  Section~\ref{sec:OT surgeries} discusses when the result of surgery can be proved to be overtwisted, and proves Theorem~\ref*{overtwistedsurgery}.

\subsection{Acknowledgements}

The author would like to thank John Etnyre for his support and many helpful discussions throughout this project.  He would like to thank Kenneth Baker for helpful discussions that led to the results in Section~\ref{sec:OT surgeries}.  The author would further like to thank Bülent Tosun, David Shea Vela-Vick, and an anonymous referee, who made helpful comments on early drafts of this paper.  This work was partially supported by NSF Grant DMS-13909073.

\section{Contact and Transverse Surgery}
\label{sec:LegendrianTransverseSurgery}

In this section, we will give a background to the contact geometric concepts used throughout the paper.  We will then describe contact and transverse surgery, and see that all contact surgeries can be re-written in terms of transverse surgeries.

\subsection{Background}

We begin with a brief reminder of standard theorems about contact structures on $3$-manifolds which we will use throughout this paper.  Further details can be found in \cite{Etnyre:contactlectures,Etnyre:contactlectures1}.

\subsubsection{Farey Graph}\label{sec:farey}
The Farey graph is the $1$-skeleton of a tessellation of the hyperbolic plane by geodesic triangles shown in Figure~\ref{fig:farey}, where the endpoints of the geodesics are labeled.  Our convention is to use cardinal directions (specifically, North, East, South, and West) to denote points on the circle.  These names will always refer to their respective locations, even as the labeling will move around.  There is a standard labeling, shown in Figure~\ref{fig:farey}, which is as follows: let West be labeled $\infty = 1/0$ and East be labeled $0$.  The third unlabeled point of a geodesic triangle with two corners already labeled $a/b$ and $c/d$ is given the label $(a+c)/(b+d)$.  We denote this \textit{Farey sum} operation $a/b \oplus c/d$.  In the Southern hemisphere, by treating $0$ as $0/1$, labeling the corners of geodesic triangles using the Farey sum suffices to label each endpoint of every geodesic with a positive number.  Thus South gets labeled $1 = 1/1$, and so on. In the Northern hemisphere, we label with negative numbers by treating $0$ as $0/(-1)$.  Thus North gets labeled $-1 = 1/(-1)$, and so on.  Every rational number and infinity is found exactly once as a label on the Farey graph.

\begin{remark} For notation purposes, we denote the result of the Farey sum of $a/b$ and $c/d$ by $a/b \oplus c/d$.  We also denote $r\oplus \cdots \oplus r$, where there are $n$ copies of $r$, as $n\cdot r$.  In general, we use the $\oplus$ operation with two fractions even when they are not connected by a geodesic on the Farey graph.  While in this case, there is no direct relation to operations on the Farey graph, we note that if points labeled $r$ and $s$ are connected by a geodesic in the graph, such that $r$ is counter-clockwise of $s$ and $rs > 0$, then $n\cdot r \oplus m \cdot s$ is the label of some point in between those labeled $r$ and $s$ and counter-clockwise of $s$. \end{remark}

Other labelings of the Farey graph can be created by labeling any two points connected by a geodesic with 0 and $\infty$.  We can then construct a labeling by an analogous process to the above.  In fact, this labeling is the result of a diffeomorphism applied to the hyperbolic plane with the standard labeling, but we will not need that fact explicitly here.

\begin{remark} If we consider each label as a reduced fraction in the obvious way (where the denominator of a negative number is negative), then each fraction naturally corresponds to a vector $$\frac{p}{q} \leftrightarrow \vect{q}{p}.$$  The vectors corresponding to $\infty$ and $0$, with which we started our labeling, constitute a basis for $\Z^2$. Hence the vectors corresponding to the endpoints of any geodesic in the graph form a basis for $\Z^2$. \end{remark}

\begin{figure}[htbp]
\begin{center}
\vspace{1cm}
\begin{overpic}[scale=1.7,tics=20]{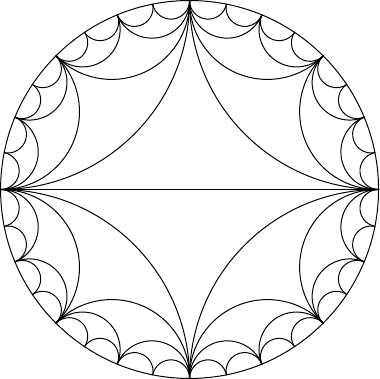}
\put(-13,153){\LARGE $\infty$}
\put(312,152){\LARGE $0$}
\put(153,-11){\LARGE $1$}
\put(147,312){\LARGE $-1$}
\put(29,267){\LARGE $-2$}
\put(-5,216){\LARGE $-3$}
\put(84,314){\LARGE $\displaystyle\frac{3}{-2}$}
\put(267,270){\LARGE $\displaystyle\frac{1}{-2}$}
\put(300,215){\LARGE $\displaystyle\frac{1}{-3}$}
\put(213,307){\LARGE $\displaystyle\frac{2}{-3}$}
\put(267,28){\LARGE $\displaystyle\frac{1}{2}$}
\put(302,87){\LARGE $\displaystyle\frac{1}{3}$}
\put(215,-7){\LARGE $\displaystyle\frac{2}{3}$}
\put(88,-7){\LARGE $\displaystyle\frac{3}{2}$}
\put(38,36){\LARGE $2$}
\put(3,90){\LARGE $3$}
\end{overpic}
\vspace{.2cm}
\caption{The Farey graph}
\label{fig:farey}
\end{center}
\end{figure}
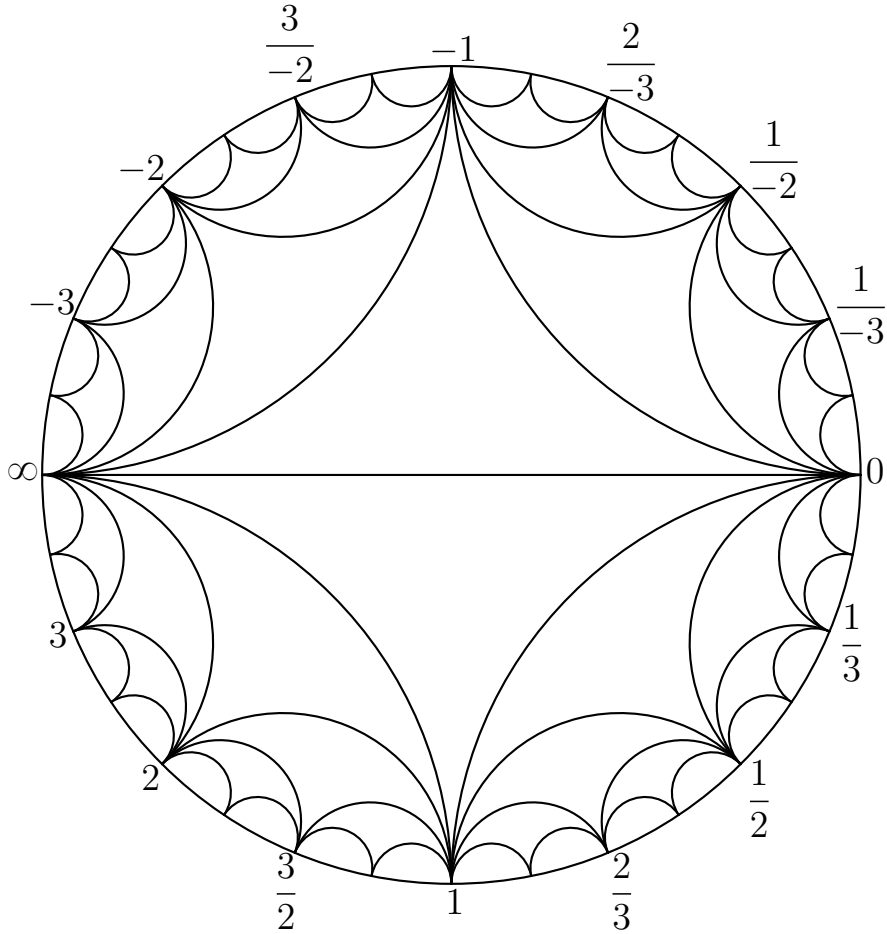

\subsubsection{Negative Continued Fractions}

We can write any negative rational number $r$ as a negative continued fraction (\textit{negative} because we are subtracting each successive fraction instead of the usual addition) in the following form

$$r = a_1+1 - \cfrac{1}{a_2-\cfrac{1}{a_3-\cfrac{1}{\cdots-\frac{1}{a_n}}}},$$ where $a_i \leq -2$ are integers.  We have $a_1 + 1$ so that other results are easier to state.  We denote this negative continued fraction as $r = [a_1+1, a_2, \ldots, a_n]$.  Given any $r$, we can create this decomposition by letting $a_1 + 1 = \floor{r}$; we then set $r' = 1/(a_1+1-r) = [a_2, \ldots, a_n]$, so $a_1 = \floor{r'}$, and so on.

\begin{remark} If $a_n=-1$, then $[a_1+1,a_2,\ldots,a_{n-1},-1] = [a_1+1,a_2,\ldots,a_{n-1}+1]$. \end{remark}

\subsubsection{Classification of Tight Contact Structures on $S^1 \times D^2$ and $T^2 \times [0,1]$}

We will use the following classifications to help define contact surgery.  The number of distinct tight contact structures will correspond to possible choices in performing contact surgery.

\begin{thm}[Kanda \cite{Kanda}]\label{solid torus} Fix a singular foliation on the boundary of $M = S^1 \times D^2$ that is divided by two dividing curves isotopic to the core of $M$.  Then there is a unique tight contact structure on $M$ up to isotopy inducing this singular foliation on the boundary.\end{thm}

Consider the manifold $(T^2 \times I,\xi)$, with $\xi$ tight. Let the two boundary components be convex with two dividing curves each, with slopes $s_0$ and $s_1$.  If $s_0$ and $s_1$ are labels on the Farey graph connected by a geodesic, then the contact manifold is called a \textit{basic slice}, or a \textit{bypass layer}.  If not, then the manifold can be cut up into basic slices along boundary parallel convex tori, following the path between $s_0$ and $s_1$ along the Farey graph.

\begin{thm}[Honda \cite{Honda:classification1}] Fix a singular foliation on the boundary of $M = T^2 \times I$  such that it is divided by two dividing curves on $T^2 \times \{i\}$ for $i = 0,1$ each of slope $s_i$.  There is a diffeomorphism of the Farey graph such that $s_0$ is sent to $-1$ and $s_1$ to $-\frac{p}{q}$, with $p>q>0$.  Let $[a_1+1,a_2, \ldots, a_n]$ be the negative continued fraction expansion of $-\frac{p}{q}$.  Then there are exactly $$|(a_1+2)(a_2+1)\cdots(a_{n-1}+1)a_n|$$ minimally twisting tight contact structures on $M$ (which means that on every convex torus parallel to the boundary of $M$, the slope of the dividing curves is between $s_0$ and $s_1$).  \end{thm}

\begin{remark}\label{remark:basic slice} If $(T^2\times I,\xi)$ is a basic slice, then there are exactly two tight contact structures up to isotopy.  These can be distinguished by their relative Euler class, and after picking an orientation, we can call them \textit{positive} and \textit{negative} basic slices; this orientation is chosen such that when gluing a negative (resp.\ positive) basic slice to the boundary of the complement of a regular neighbourhood of a Legendrian knot, the result is the complement of a regular neighbourhood of its negative (resp.\ positive) stabilisation. \end{remark}

We can denote a choice of tight contact structure on $T^2 \times I$ with dividing curve slopes $s_0$ and $s_1$ by a choice of sign (positive or negative) on each jump of the shortest path on the Farey graph between $s_0$ and $s_1$.  Note that distinct sign assignations can give isotopic contact structures.  We say that the path between $s_0$ and $s_1$ can be \textit{shortened} if there is a sequence of labels $c_i, c_{i+1}, \ldots, c_{j-1},c_j$ on the path such that $c_i$ and $c_j$ are connected by a geodesic on the Farey graph.  If such a sequence exists, and if the signs on the jumps $c_i \to c_{i+1}$, \ldots, $c_{j-1} \to c_j$ are inconsistent, then after gluing the basic slices with boundary slopes $c_k$ and $c_{k+1}$ (for $k = i, \ldots, j-1$), we get a basic slice with boundary slopes $c_i$ and $c_j$, where the contact structure is not isotopic to either of the two tight contact structures guaranteed by Remark~\ref{remark:basic slice}, and thus the induced contact structure is overtwisted.

\subsection{Contact Surgery}

We take a moment here to give conventions for topological Dehn surgery on an oriented knot $K$ in a $3$-manifold $M$.  Given any longitude $\lambda$ on the boundary of a regular neighbourhood $N$ of $K$, and the meridian $\mu$ of $K$, we define the result $M_{p/q}(K)$ of $p/q$-surgery on $K$ to be the result of removing the interior of $N$ and gluing in a $D^2 \times S^1$ (sometimes called the \textit{surgery torus}) such that the curves parallel to $p\mu + q\lambda$ on $\bd(M\backslash N)$ bound a $D^2$ in $D^2 \times S^1$.  The \textit{(surgery) dual knot} to $K$ from this surgery is the image of the core of $D^2 \times S^1$ in $M_{p/q}(K)$.

Now, given a contact manifold $\Mxi$ and a knot $L$ in $M$, $L$ is \textit{Legendrian} if its tangent vector always lies in $\xi$.  All Legendrian knots in this paper are oriented.

The framework for surgery on Legendrian knots, termed \textit{contact surgery}, arose out of classic work of Eliashberg \cite{Eliashberg:stein}, with a modern description provided by work of Kanda \cite{Kanda} and Honda \cite{Honda:classification1}.  This description of contact surgery is based on work of Honda \cite{Honda:classification1}.

\begin{definition} In this paper, whenever surgery is performed in the contact category on a Legendrian knot, we use the adjective ``contact'', and in \textit{contact $r$-surgery}, the $r$ is with respect to the contact framing; thus, contact $(-1)$-surgery is topologically Dehn surgery with framing 1 less than the contact framing.  In some contact geometry papers, parentheses are used to distinguish surgeries with respect to contact framings from the standard topological surgery notation.  This paper does not make such a distinction; we will be using surgery coefficients sufficiently complicated to warrant parentheses, and we do not wish to add ambiguities by having parentheses have special meaning.

Given a Legendrian knot $L \subset \Mxi$, we first remove a standard neighbourhood of $L$, \textit{ie.\@} a tight solid torus with convex boundary, where the dividing curves have the same slope as the contact framing $f$ (when $L$ is null-homologous, $f = tb(L)\mu + \lambda$, where $\mu$ is a meridian and $\lambda$ is the Seifert framing of $L$).  To do \textit{contact $r$-surgery}, where $r$ is any non-zero rational number, we first choose the shortest counter-clockwise path from $0$ to $r$ in the Farey graph.  Let the labels along this path be $c_1=0, \ldots, c_k = r$.  For each jump $c_1\to c_2,\ldots,c_{k-2}\to c_{k-1}$, we glue in a basic slice of either sign ($+$ or $-$), where the back face of the basic slice corresponding to $c_i \to c_{i+1}$ has dividing curves of slope $f+c_i$ and the front face has dividing curves of slope $f+c_{i+1}$.  Finally, glue in a solid torus with meridional slope $f+c_k$ and with dividing curves on the boundary of slope $f+c_{k-1}$ (which according to Theorem~\ref{solid torus} has a unique tight contact structure up to isotopy).  The result is $(M_{f+r}(L),\xi_{f+r})$.

Contact surgery with negative (resp.\ positive) framing (relative to the natural contact framing) is called \textit{negative (resp.\ positive) contact surgery} (negative contact surgery is often called \textit{Legendrian surgery} in the literature).  Note that many choices were made along the way, and in general there is no well-defined contact surgery.  When $r = 1/m$ for some integer $m \neq 0$, however, there is a unique contact surgery.  When $r > 1$ is an integer, there are two possibilities, depending on the sign of the single basic slice used to perform the surgery.  In general, there are two preferred possibilities: one where all negative signs are chosen, and one where all positive signs are chosen.  We denote these two possibilities by $\xi^-_r$ and $\xi^+_r$.  When not otherwise indicated, contact $r$-surgery will refer to $\xi^-_r$. \end{definition}

We now describe Ding and Geiges's algorithm \cite{DG:surgery} for converting a general contact surgery diagram into one only involving contact $(\pm 1)$-surgeries.

\begin{const}\label{DingGeiges}
Let $L$ be a Legendrian knot, $r < 0$ be a rational number, and let $[a_1 + 1, a_2, \ldots, a_n]$ be its negative continued fraction decomposition.  We define $L_1 \cup \cdots \cup L_n$ as follows: $L_1$ is a Legendrian push-off of $L$, stabilised $|a_1 + 2|$ times, and $L_i$ is a Legendrian pushoff of $L_{i-1}$, stabilised $|a_i + 2|$ times, for $i = 2, \ldots, n$.  We perform contact $(-1)$-surgery on each $L_i$.  This is equivalent to a contact $r$-surgery on $L$.  The choice of stabilisation is important, but we distinguish two natural choices: choosing all negative stabilisations gives $\xi^-_r$, whereas choosing all positive stabilisations gives $\xi^+_r$ on $M_{tb(L)+r}(L)$.  These choices are equivalent to choices of signs on basic slice layers involved in the surgery.

If $r = p/q > 0$, let $n$ be a positive integer such that $1/n \leq r< 1/(n-1)$.  We take $n$ push-offs of $L$, and do contact $(+1)$-surgery on them, and do contact $r'$-surgery on $L$, where $r' = p/(q-np) < 0$.  Again, we have a choice of stabilisations for the contact $r'$-surgery, and we label the two distinguished choices as above.
\end{const}

\subsection{Transverse Surgery}

Given a contact manifold $\Mxi$ and a knot $T$ in $M$, $T$ is \textit{transverse} if its oriented tangent vector is always positively transverse to $\xi$ (note that we require an orientation on $T$ for this to make sense).  Work of Gay \cite{Gay:transverse} and Baldwin and Etnyre \cite{BE:transverse} provide a natural set-up for transverse surgery, building off of classic constructions of Lutz \cite{Lutz:s3,Lutz:circlebundles} and Martinet \cite{Martinet}.  Transverse surgery comes in two flavours: \textit{admissible} (``removing twisting'' near the knot) and \textit{inadmissible} (``adding twisting'' near the knot).  This description of transverse surgery follows Baldwin and Etnyre \cite{BE:transverse}.

\begin{definition}\label{def:transverse surgery}
Given a transverse knot $K\subset \Mxi$, a standard neighbourhood is contactomorphic to $S^1 \times \{r \leq a\}$ in $(S^1 \times \R^2, \xi_{\rm{rot}}=\ker(\cos r\, dz + r\sin r\,d\theta))$ for some $a$, where $z$ is the coordinate on $S^1$, $K$ is identified with the $z$ axis, and some framing of $K$ is identified with $\lambda = S^1 \times \{r = a,\theta = 0\}$.  The characteristic foliation on the torus $\{r = r_0\}$ is given by parallel lines of slope $-\cot r_0 / r_0$ (\textit{ie.\@} $r_0\lambda -\cot(r_0)\mu$, where $\mu$ is a meridian of $K$).  Note that there is not a unique $r_0$ corresponding to a given slope.  A neighbourhood of $K$ is identified with the solid torus $S^n_a$ bounded by the torus $$T^n_a = \{r = r_0\,\,|\,\,\mbox{$r_0$ is the $n$th smallest positive solution to }-\frac{\cot r_0}{r_0} = a\},$$ for some slope $a \in \R\cup \{\infty\}$ and some positive integer $n$.  To perform \textit{admissible transverse surgery}, we take a torus $T^m_b$ inside $T^n_a$, where $m\leq n$ and if $m = n$, then $b < a$.  We remove the interior of $S^m_b$ from $S^n_a$, and perform a contact cut on the boundary.  See Lerman \cite{Lerman} and Baldwin and Etnyre \cite{BE:transverse} for details. This gives us a smooth manifold $(M_b(K),\xi^m_b)$ with a well-defined contact structure.  To perform \textit{inadmissible transverse surgery}, consider the open manifold obtained by removing the knot $K$.  We can take the closure of this manifold to get one with torus boundary, and uniquely extend the contact structure so that along the boundary, the characteristic foliation is given by curves of slope $-\infty$, that is, meridional slope (the minus sign is not strictly needed, but it serves as a reminder that the characteristic foliations of nearby boundary-parallel tori have large negative slope).  This is the inverse operation of a contact cut.  We then glue on a $T^2 \times I$ with a contact structure modeled  on $\xi_{\rm{rot}}$ above, such that the contact planes twist out to some slope $b$ (we can add $n$ half-twists of angle $\pi$ before stopping at $b$, and there is no restriction on $b$).  After performing a contact cut on the new boundary, we are left with the manifold $(M_b(K),\xi^n_b)$. Unless otherwise stated, we will assume $n=0$ and omit $n$. \end{definition}

\begin{remark} Although inadmissible transverse surgery is well-defined, the result of admissible transverse surgery in general depends on the torus neighbourhood.  For example, Etnyre, LaFountain, and Tosun \cite{ELT:cables} show that the $(2,3)$-torus knot in $(S^3,\xi_{\rm{std}})$ has infinitely many distinct standard neighbourhoods which are not subsets of each other.  It can be shown that for each rational number $\frac{1}{n+1} < r < \frac{1}{n}$, there are $n$ distinct results of admissible transverse surgery. \end{remark}

Recall that Baldwin and Etnyre \cite{BE:transverse} showed that contact $(-1)$-surgery on a Legendrian $L$ is equivalent to an admissible transverse surgery on the positive transverse push-off of $L$.  We show the equivalent result for positive contact surgeries.

\begin{prop}\label{prop:+1 is transverse} Let $L$ be a Legendrian knot in some contact manifold.  Then the contact manifold obtained via contact $(+1)$-surgery on $L$ can also be obtained via an inadmissible transverse surgery on a transverse push-off $K$ of $L$.
\end{prop}
\begin{proof}
Let $L$ be a Legendrian knot in $(M,\xi)$ and $N$ a standard neighbourhood of $L$.  We will work in the contact framing, where $\mu$ is a meridian and $\lambda$ is the contact framing.  A local model for $N$ is a tight solid torus $S$ obtained from $S_0$ by perturbing the boundary to be convex with dividing set consisting of two parallel curves of slope $0$.  We know that $N$ and $S$ are contactomorphic, so we only need consider $S$.

Note that $S$ is contactomorphic to $S \cup S'$, where $S'$ is an $I$-invariant collar neighbourhood of $\partial S$.  Contact $(+1)$-surgery on $L$ is performed by removing $S$ from $S \cup S'$ and gluing $S$ back in according to the map $\bigl(\begin{smallmatrix}1&1\\ 0&1\end{smallmatrix} \bigr)$.  The resulting manifold is a tight solid torus whose dividing curves have slope $0$ and whose meridian has slope $+1$, when measured with respect to the original contact framing.  By Theorem~\ref{solid torus}, there is a unique tight contact structure on such a solid torus up to isotopy.

Consider now a transverse push-off $K$ of $L$ in $S$.  It suffices to show that inadmissible transverse 1-surgery on $K$ results in a tight solid torus contactomorphic to the one produced by contact $(+1)$-surgery.  Note that $S$ is a standard neighbourhood for $K$.  Consider the result of inadmissible transverse $1$-surgery on $K$ in $S$.  The new solid torus is a standard neighbourhood of the dual knot, \textit{ie.\@} the core of the surgery torus, and so it is tight.  The boundary is still convex, and the dividing curves still have slope $0$ when measured with respect to the original contact framing.  The meridian has slope $+1$, and thus by Theorem~\ref{solid torus}, the tight contact structure on the surgery torus coming from inadmissible transverse $1$-surgery agrees with that coming from contact $(+1)$-surgery.
\end{proof}

\begin{proof}[Proof of Theorem~\ref{thm:transverse surgery gets everything}]
Ding and Geiges \cite{DG:surgery} showed that every contact 3-manifold can be obtained by contact $(\pm 1)$-surgery on a link in $S^3$.  Combing Baldwin and Etnyre's result \cite{BE:transverse} with Proposition~\ref{prop:+1 is transverse} gives us the result. \end{proof}

\begin{remark} In Section~\ref{sec:OBD}, we will prove Theorem~\ref{Transverse=Legendrian}, showing that inadmissible transverse surgery on a transverse knot agrees with contact $r$-surgery on a Legendrian approximation, for some positive $r$.  Our proof will be via open book decompositions, hence we defer it until Section~\ref{sec:OBD}.  However, one could prove it directly from the definition, considering the signs of basic slices used to create the surgery torus in either instance. \end{remark}

Before we leave this section, we give a lemma that will be used throughout this paper.  Compare \cite[Proposition~7]{DG:surgery}, where they prove the corresponding result for surgeries on Legendrian knots.

\begin{lem}\label{inadmissible then admissible}
Any inadmissible transverse $r$-surgery on a transverse knot $K\subset \Mxi$ with $r > 0$ can be obtained by an inadmissible transverse $(1/m)$-surgery on $K$ followed by some admissible transverse surgery on the surgery dual knot to $K$, where $m$ is a positive integer such that $1/m < r$.
\end{lem}
\begin{proof}
Say we wish to perform inadmissible transverse $r$-surgery including $n$ half-twists; the resulting contact manifold is $(M_r(K),\xi^n_r)$.  In performing inadmissible transverse $(1/m)$-surgery on $K$ with exactly $n$ half-twists, we first remove $K$ and take the closure, such that the characteristic foliation on the boundary has leaves of slope $-\infty$.  We then glue on a $T^2 \times I$, where the contact structure rotates from having having leaves of slope $-\infty$ through $n$ half-twists and then beyond from a torus $T_1$ with leaves of slope $-\infty$ to the new boundary $T_2$, which has leaves of slope $1/m$.  There exists a boundary-parallel torus $T_3$ with characteristic foliation of leaves of slope $r$ in between $T_1$ and $T_2$.

After performing a contact cut on the new boundary $T_2$, the torus $T_1$ bounds a standard neighbourhood of the dual knot $K'$ to $K$, contactomorphic to $S^0_a$, for some rational number $a$ (in fact, it is not hard to see that $a = -1/m$, if we choose the same longitude for $K'$ as for $K$).  In this contactomorphism, the image of $T_3$ is $T^0_b$ for some rational number $b < a$, and it bounds $S^0_b$.  Thus, in performing admissible transverse $b$-surgery on $K'$, we remove the interior of $S^0_b$, and perform a contact cut on the boundary.  However, removing the interior of $S^0_b$ leaves us with the same manifold as we would get by gluing a $T^2 \times I$ layer to the complement of $K$ in $M$ that stopped at $T_3$ instead of going out to $T_2$.  Thus, since the contact cuts are these manifolds are the same, originally performing inadmissible transverse $r$-surgery (with $n$ half-twists) on $K$ results in the same contact manifold as admissible transverse $b$-surgery on $K'$.
\end{proof}

\section{Open Book Decompositions}
\label{sec:OBD}

In this section, we discuss what happens when we perform transverse surgery on the binding component of an open book decomposition.  We create integral open books that support the contact manifolds resulting from any inadmissible transverse surgery, and from admissible surgeries where the surgery coefficient is smaller than the page slope.  At this point we are unable to create open books for admissible surgery larger than the page slope.

\subsection{Background}

We recall the definition of an open book, and what it means for an open book to support a contact structure.  See \cite{Etnyre:OBlectures} for more background.

A manifold $M$ has an \textit{(integral) open book} $\Sigphi$, for $\Sigma$ a surface with non-empty boundary, and $\phi \in \textrm{Diff}^+(\Sigma, \bd\Sigma)$, if $M$ is diffeomorphic to $$\frac{\Sigma \times [0,1]}{(x,1) \sim (\phi(x),0)} \bigcup \left(\bigcup \left(S^1 \times D^2\right)\right),$$ where there are as many $S^1 \times D^2$ factors as there are boundary components of $\Sigma$, and each solid torus is glued such that $S^1\times \{* \in \bd D^2\}$ gets mapped to a distinct boundary component of $\Sigma \times \{*\}$ and $\{* \in S^1\} \times \bd D^2$ gets mapped to $\{* \in \bd\Sigma\} \times S^1$.
Each $\Sigma$ admits an extension so that its boundary lies on the cores of the glued-in $S^1 \times D^2$; we also call this extension $\Sigma$.  We call each copy of $\Sigma$ a \textit{page}, and the image of $\bd\Sigma$ (\textit{ie.\@} the cores of the glued-in tori) is called the \textit{binding}.  We say that $\Sigphi$ \textit{supports} a contact structure $\Mxi$ if there exists a contact $1$-form $\alpha$ for $\xi'$ isotopic to $\xi$ such that:

\begin{itemize}
\item $d\alpha$ is a positive area form on each page,
\item $\alpha > 0$ on the binding.
\end{itemize}

There is a unique contact structure supported by a given open book \cite{TW,Giroux:OBD}.  We also have an operation called \textit{positive stabilisation}, which creates a new open book supporting the same manifold with the same contact structure.  This operation takes $\Sigphi$ to $(\Sigma', \phi')$, where $\Sigma'$ is $\Sigma$ plumbed with an annulus, and $\phi' = \phi \circ D_\gamma$, where $D_\gamma$ is a positive Dehn twist along the core of the annulus.  The core of the annulus intersected with the original page $\Sigma$ defines an arc, called the \textit{stabilisation arc}.  The stabilised open book supports the same manifold and the same contact structure as did $\Sigphi$.  With this, we can formulate:

\begin{thm}[Giroux \cite{Giroux:OBD}] Let $M$ be a closed $3$-manifold.  Then there is a bijection between

$$\frac{\{\textrm{open books of }M\}}{\textrm{positive stabilisation}}\mbox{\hspace{.5cm} and \hspace{.5cm}}\frac{\{\textrm{contact structure on }M\}}{\textrm{contactomorphism}}.$$\end{thm}

Baker, Etnyre, and van Horn-Morris \cite{BEVHM} have also formulated a notion of a \textit{rational open book}, where the gluing of the solid tori results in the pages approaching the binding in a curve that may not be a longitude.  In general, the binding may be a rationally null-homologous link in $M$.  They proved that a rational open book, just like its integral cousin, supports a unique contact structure up to isotopy.

\subsection{Surgery on Binding Components of Open Books}

Given an open book $\Sigphi$, and a knot $K$ in the binding, we define the \textit{page slope} to be the curve traced out by the boundary of the page on $\bd N$, where $N$ is a neighbourhood of $K$.  When $\Sigphi$ is an integral open book, this gives rise to a coordinate system on $\bd N$, or equivalently a framing of $K$, where the longitude $\lambda$ is chosen to be the page slope.  We orient $K$ as the boundary of $\Sigma$, and we let $\mu \subset \bd N$ be an oriented meridian of $K$ linking it positively.  This gives us coordinates $(\lambda, \mu)$ on $\bd N$.  The curve $p\mu + q\lambda$ corresponds to the slope $p/q$, so the page slope is slope 0. Any negative slope with respect to this framing is an admissible surgery slope (\textit{cf.\@} \cite[Lemma 5.3]{BE:transverse}).  If $\Sigphi$ is an integral open book (or at least, if it is integral at the binding component $K$), then doing Dehn surgery on $K$ with a negative surgery coefficient (with respect to the page slope) induces an rational open book on the resulting manifold.  We have the following.

\begin{thm}[Baker--Etnyre--van Horn-Morris \cite{BEVHM}]\label{BEV OBD Thm} The induced rational open book from Dehn surgery on $K$ less than the page slope supports the contact structure coming from admissible surgery on the transverse knot $K$. \end{thm}

In this section, we prove Theorem~\ref{OBDThm}, which is an extension of Theorem~\ref{BEV OBD Thm} to positive surgeries.  We begin with recalling the following proposition, see \cite[Theorem 5.7]{Etnyre:OBlectures}.

\begin{prop}\label{prop:plusminus one surgery} Let $K$ be a knot on the page of an open book $\Sigphi$ for $\Mxi$.  The open book $(\Sigma, \phi \circ D_K^{\pm 1})$ is an open book for the manifold given by $\mp 1$-surgery on $K$, with respect to the framing given by $\Sigma$, where $D_K$ is a positive Dehn twist about a curve on $\Sigma$ isotopic to $K$.  Moreover, if $K$ is Legendrian, then this open book supports the contact structure coming from contact $(\mp 1)$-surgery on $K$.\end{prop}

To prove Theorem~\ref{OBDThm}, we split the process into the following steps.  First, we prove that for $n > 0$ an integer, the open book coming from $1/n$ surgery on the boundary component corresponding to $K$ supports the result of inadmissible transverse $1/n$-surgery on $K$.  Then we note that all other inadmissible slopes greater than $0$ can be achieved by a combination of an inadmissible transverse $1/n$-surgery followed by an admissible surgery, and invoking Theorem~\ref{BEV OBD Thm} completes the proof.  We will first give a proof of the first step when $\Sigma$ has multiple boundary components (Theorem~\ref{thm:OBD multiple boundaries}) and then a proof for an arbitrary number of boundary components (Theorem~\ref{thm:OBD connected binding}).

\begin{thm}\label{thm:OBD multiple boundaries} If $\Sigphi$ is an open book, $|\bd \Sigma| \geq 2$, and $n > 0$ is an integer, then the open book induced by topological $1/n$-surgery on a binding component $K$ (where the surgery coefficient is with respect to the page slope) supports the contact structure coming from inadmissible transverse $1/n$-surgery on $K$.  In particular, the new open book is $(\Sigma, \phi\circ D_K^{-n})$. \end{thm}
\begin{proof} Recall that Baldwin and Etnyre \cite[Lemma 5.3]{BE:transverse} showed that if $|\bd \Sigma| \geq 2$, then we can find a standard neighbourhood of any binding component contactomorphic to $S^1_{\varepsilon}$, for some $\varepsilon > 0$ (with respect to the page slope).  Thus we know that there is a neighbourhood of $K$, whose boundary is a pre-Lagrangian torus with slope 0, where each leaf of the foliation on the boundary is a Legendrian sitting on a page of the open book.  We know from Proposition~\ref{prop:plusminus one surgery} that contact $(+1)$-surgery on a knot on the page of an open book, which is the same as $(+1)$-surgery with respect to the page slope, is equivalent to changing the monodromy of the open book by a negative Dehn twist.  We know from Proposition~\ref{prop:+1 is transverse} and Ding and Geiges's description \cite{DG:surgery} of contact surgery that doing contact $(+1)$-surgery on $n$ leaves of slope $0$ is equivalent to doing inadmissible transverse $1/n$-surgery on $K$.  The result follows. \end{proof}

\begin{proof}[Proof of Theorem~\ref{OBDThm} for $|\bd\Sigma| \geq 2$] Given an inadmissible slope $p/q > 0$, choose a positive integer $n$ such that $1/n < p/q$.  By Lemma~\ref{inadmissible then admissible}, we know that inadmissible transverse $p/q$-surgery is equivalent to doing inadmissible transverse $1/n$-surgery, followed by some admissible transverse surgery.  Combining Theorem~\ref{thm:OBD multiple boundaries} and Theorem~\ref{BEV OBD Thm} gives the result.\end{proof}

In order to extend this result to the case of a single boundary component, we construct an explicit contact form that is supported by the induced open book and whose kernel is isotopic to the contact structure of inadmissible transverse surgery.

\begin{thm}\label{thm:OBD connected binding} If $\Sigphi$ is an open book (with any number of binding components), then the open book induced by topological $1/n$-surgery on a binding component $K$ (where the surgery coefficient is with respect to the page slope) supports the contact structure coming from inadmissible transverse $1/n$-surgery on $K$.  In particular, the new open book is $(\Sigma, \phi\circ D_K^{-n})$, where $D_K$ is a positive Dehn twist about a curve parallel to the binding component $K$. \end{thm}
\begin{proof}
We need to show that the contact structure $\xi_{1/n} = \ker \alpha_{1/n}$ coming from inadmissible transverse $1/n$-surgery on $K$ is supported by the open book coming from $1/n$-surgery on $K$.  Notice that outside a neighbourhood of $K$, $\xi_{1/n}$ agrees with $\xi$, the contact structure supported by $\Sigphi$, and so we only need to look at a neighbourhood of $K$.  Let the page slope be the curve $\lambda$.  Then after surgery, we have a coordinate system $(\lambda',\mu')$, where $\lambda' = \lambda$ and $\mu' = n\lambda + \mu$.

There is a regular neighbourhood of $K$ such that the contact form on $S^1 \times D^2$ is given by $f(r)\,dz + g(r)\,d\theta$, where $z \sim z+1$ is the coordinate on $S^1$, $(f,g)$ is $(1,r^2)$ near $r = 0$, and as $r$ increases, $f$ tends to $0$ (in particular, we can use $S^0_0$ from Definition~\ref{def:transverse surgery}).  The change of coordinates coming from surgery gives a map $\Phi : S^1 \times D^2 \to S^1 \times D^2$, which sends $(z, r, \theta)$ to $(z - \frac{n}{2\pi}\theta,r, \theta)$.  Think of this as sending $\mu$ to $\mu' - n\lambda'$.  Thus, we require a contact form $\beta = h(r)\,dz + k(r)\,d\theta$ on $S^1 \times D^2$ (representing a neighbourhood after surgery) for which
\begin{itemize}
\item $h(r)k'(r) - k(r)h'(r) > 0$,
\item $(h,k) = (1, r^2)$ near $r =0$,
\item $(h,k) = (f, g-\frac{n}{2\pi}f)$ for $r \geq r_1$ for some $r_1> 0$,
\item $h'(r) < 0$ for $r>0$.
\end{itemize}
The first condition is the contact condition, and the third condition is to ensure that the contact planes match the original ones for $r \geq r_1$.  These conditions are possible to meet.  Note that the third condition is possible to meet because as $r$ increases, $f$ approaches $0$, so the condition is a small perturbation of the achievable condition $(f,g)$.  This allows us to define a contact form $\alpha' = \Phi^*(\beta) = h(r)\,dz + (k(r) - \frac{n}{2\pi}h(r))\,d\theta$ on $S^1 \times D^2$ which patches in with $\alpha$ (the contact form defining $\xi$), and $\ker\alpha' = \xi_{1/n}$.

It remains to check that this contact form is supported by the open book in $S^1 \times D^2$.  We see that $\alpha'$ restricted to the binding is positive, as in the new coordinate system, $\beta$ near the binding is $dz + r^2\,d\theta$.  Notice that on a page, $\theta$ is constant, so $d\theta = 0$, and so $d\alpha'|_\mathrm{page} = h'(r)\,dr\wedge dz = -h'(r)\,dz \wedge dr > 0$.  Thus the open book supports the contact structure $\xi_{1/n}$.
\end{proof}

\begin{proof}[Proof of Theorem~\ref{OBDThm}] Given Theorem~\ref{thm:OBD connected binding}, the rest of the proof is identical to the case of multiple boundary components.\end{proof}

\subsection{Constructions of Open Books}
\label{sec:construction}

We use the Farey graph to track operations on open books is as follows.  The three operations that we perform on a binding component of an open book (corresponding to a knot $K$) are the following:
\begin{enumerate}
\item positively stabilise with a boundary parallel stabilisation arc,
\item do $(+1)$-surgery (with respect to the page slope), or
\item do $(-1)$-surgery (with respect to the page slope).
\end{enumerate}

We will always treat the West point of the Farey graph as the meridian slope, and the East point as the page slope.  Whenever one of these quantities changes (the page slope as a result of stabilisation, and the meridian slope as a result of surgery), we will adjust the labeling of the Farey graph accordingly.  Thus the labels on the Farey graph will always correspond to slopes measured with respect to the original page slope and meridian slope.  The results of the three operations are shown in Figure~\ref{fig:farey operations}.

The result of (1), positive stabilisation, is to keep the meridian fixed, but lower the page slope by 1.  This has the effect on the labeling of the Farey graph of keeping West's label fixed, moving North's label to East, East's label to South, and South's label to South-West (along with the resulting change of all other labels).  When we stabilise, we add a new boundary component together with a positive Dehn twist parallel to the new boundary (this boundary component corresponds to a new unknot component in the binding which links $K$).  When doing further operations on the binding, we ignore this new boundary component and focus on the boundary component which did not acquire a positive Dehn twist.  This boundary component is isotopic to the original knot we were considering.  The result of (2), $(+1)$-surgery on the binding, is to keep the page slope fixed, but change the meridian slope.  This has the effect on the labeling of the Farey graph of keeping East's label fixed, moving South's label to West, West's label to North, and North's label to North-East.  Finally, the result of (3), $(-1)$-surgery on the binding, is similar to $(+1)$-surgery, where East's label is fixed, but now North's label moves to West, West's label moves to South, and South's label to South-East.  For more details on the effect of stabilisation and surgery on the open book, see \cite{Etnyre:OBlectures}.

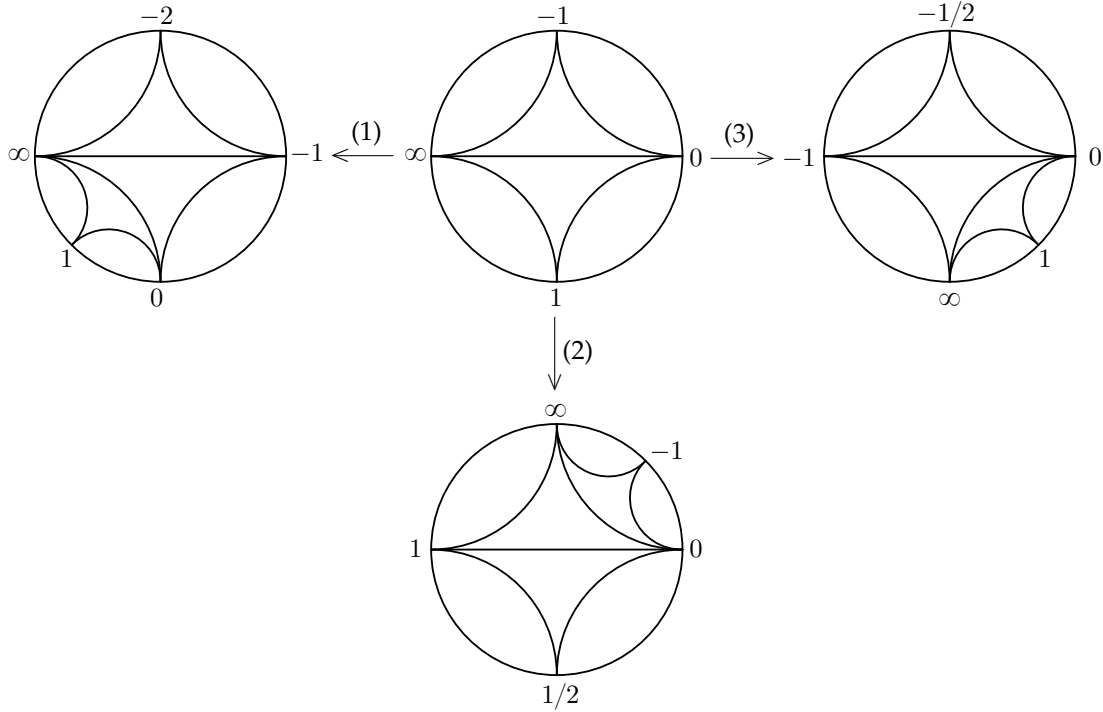
\begin{figure}[htbp]
\begin{center}
\vspace{0.2cm}
\begin{overpic}[scale=0.8,tics=20]{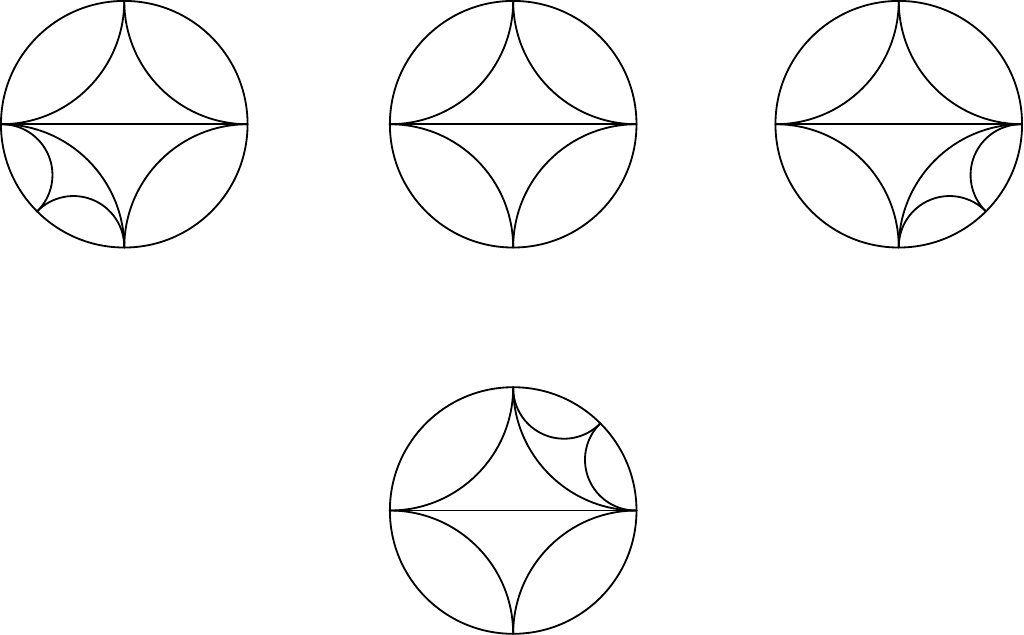}
\put(136,197){\line(-1,0){23}}
\put(112,194.5){$<$}
\put(140,195){$\infty$}
\put(97,195){$-1$}
\put(120,203){(1)}
\put(248,193){$0$}
\put(256,196){\line(1,0){23}}
\put(273,193.5){$>$}
\put(261,202){(3)}
\put(283,193){$-1$}
\put(197,136){\line(0,-1){27}}
\put(194.48,115){\rotatebox{-90}{$>$}}
\put(200,120){(2)}
\put(195,140){$1$}
\put(193,98){$\infty$}
\put(-10,195){$\infty$}
\put(399,193.5){$0$}
\put(190,247){$-1$}
\put(40,247){$-2$}
\put(334,248){$-1/2$}
\put(44,140){$0$}
\put(342,140){$\infty$}
\put(10,155){$1$}
\put(380,155){$1$}
\put(142,45){$1$}
\put(248,45){$0$}
\put(233,82){$-1$}
\put(192,-10){$1/2$}
\end{overpic}
\vspace{0.2cm}
\caption{The results of operations (1), (2), and (3) on the labels of the Farey graph, where we start with the standard labeling.  The new locations of the original labels for the four cardinal directions are shown, as well as the new labels for those four spots. The rest of the labels for the Farey graph can be inferred from the given labels, as in Section~\ref{sec:farey}.}
\label{fig:farey operations}
\end{center}
\end{figure}

\begin{example} Starting from the standard labeling, if we do a positive stabilisation followed by a $(-1)$-surgery with respect to the new page slope, East's label will read $(-1)$, West's label will read $-2$, North's label will read $-3/2$, and South's label will read $\infty$.  Thus if we modify our open book in this way, we will have performed a $-2$ surgery (with respect to the original page slope) on the binding component, and the new page slope is $(-1)$, measured with respect to the original page slope. \end{example}

\subsubsection{Open books compatible with admissible transverse surgery}

Although there are natural rational open books for transverse surgeries on a binding component, as in Theorem~\ref{OBDThm}, we would prefer to have an integral open book that supports the same contact structure as the rational open book.
We start by constructing an open book compatible with admissible transverse $r$-surgery on a binding component $K$, where $r < 0$ is with respect to the page slope.  We write $r$ as a negative continued fraction.

We first prove some easy lemmata about the Farey graph.

\begin{lem}\label{Farey sum} The rational number $[a_1+1, a_2, \ldots, a_n]$ for $a_i \leq -2$ and $n>1$ is given by $$[a_1+1,a_2, \ldots, a_n] = |a_n+1|\cdot[a_1+1,a_2, \ldots, a_{n-1}]\oplus [a_1+1,a_2,\ldots, a_{n-1}+1].$$ \end{lem}
\begin{proof}
We prove this by induction on $n$.  For $n = 2$, $$[a_1+1,a_2] = a_1+1 - \frac{1}{a_2} = \frac{a_1a_2+a_2-1}{a_2}.$$  We calculate,
\begin{align*}
|a_2+1|\cdot[a_1+1]\oplus[a_1+1+1] &= (-a_2-1)\cdot\left(\frac{-1-a_1}{-1}\right)\oplus\left(\frac{-2-a_1}{-1}\right) \\
&= \left(\frac{a_1a_2+a_1+a_2+1}{a_2+1}\right)\oplus\left(\frac{-2-a_1}{-1}\right) \\
&= \frac{a_1a_2+a_2-1}{a_2} = [a_1+1,a_2],
\end{align*}
where the abundant negative signs are because $a_i \leq -2$, and the denominator of the fraction associated to a negative number on the Farey graph must be negative.

For $n > 2$, we consider the numbers $$\frac{p}{q} = [a_2, \ldots, a_{n-1}]$$ and $$\frac{p'}{q'} = [a_2, \ldots, a_{n-1}+1].$$  From our induction hypothesis, we know that $$\frac{r}{s} = [a_2, \ldots, a_n] = |a_n+1|\cdot [a_2, \ldots, a_{n-1}] \oplus [a_2, \ldots, a_{n-1}+1] = |a_n+1|\cdot\left(\frac{p}{q}\right)\oplus\left(\frac{p'}{q'}\right) = \frac{|a_n+1|p+p'}{|a_n+1|q+q'}.$$  We now adjust the left hand side so that it represents the number we want, that is $$[a_1+1, a_2,\ldots, a_n] = a_1 + 1 - \frac{1}{r/s} = \frac{a_1r+r - s}{r}.$$  We then can calculate
\begin{align*}
|a_n+1|\cdot[a_1+1,a_2,\ldots,a_{n-1}]&\oplus[a_1+1,a_2,\ldots,a_{n-1}+1] \\
&= |a_n+1|\left(a_1+1-\frac{1}{p/q}\right) \oplus \left(a_1+1 - \frac{1}{p'/q'}\right) \\
&= |a_n+1|\cdot\left(\frac{a_1p+p-q}{p}\right)\oplus\left(\frac{a_1p'+p'-q'}{p'}\right) \\
&= \frac{|a_n+1|a_1p+|a_n+1|p-|a_n+1|q}{|a_n+1|p}\oplus\frac{a_1p'+p'-q'}{p'} \\
&= \frac{(|a_n+1|pa_1+p'a_1)+(|a_n+1|p+p')-(|a_n+1|q+q')}{|a_n+1|p+p'} \\
&= \frac{a_1r+r-s}{r} = [a_1+1,a_2,\ldots,a_n].
\end{align*}
\end{proof}

\begin{lem}\label{Farey connected} Points on the standard Farey graph with labels $[a_1+1,a_2, \ldots, a_n]$ and $[a_1+1,a_2, \ldots, a_n+1]$, where $a_i \leq -2$, are connected by a geodesic. \end{lem}
\begin{proof}
We induct on $n$.  When $n=1$ this is clear.  When $n > 1$, we see from Lemma~\ref{Farey sum} that $$[a_1+1,a_2,\ldots,a_n] = |a_n+1|\cdot[a_1+1,a_2,\ldots,a_{n-1}]\oplus[a_1+1,a_2,\ldots,a_{n-1}+1].$$ Since by our inductive hypothesis $[a_1+1,a_2,\ldots,a_{n-1}]$ and $[a_1+1,a_2,\ldots,a_{n-1}+1]$ are connected by a geodesic, we see also that the three numbers $[a_1+1,a_2,\ldots,a_{n-1}]$, $[a+1,a_2,\ldots,a_{n-1},a_{n-1}+1]$ and $[a_1+1,a_2,\ldots,a_{n-1},-2]$ form the endpoints of a geodesic triangle, as the latter term is the sum of the two former terms, by Lemma~\ref{Farey sum}.  We now add $|a_n+1|-1$ additional copies of $[a_1+1,a_2,\ldots,a_{n-1}]$ to $[a_1+1,a_2,\ldots,a_{n-1},-2]$ to arrive at $[a_1+1,a_2,\ldots,a_n]$.  The proof is concluded by noting that as long as two numbers $a$ and $b$ are connected by a geodesic, so are $b$ and $a\oplus b$.
\end{proof}

In the following, we will be stabilising and changing the monodromy of an open book.  Throughout this process, we will track a single boundary component, calling it $K$ at every stage, even though we will be changing the manifold.  When stabilising the boundary component $K$ along a boundary-parallel stabilisation arc, the result has two boundary components where previously there was only one; the boundary component without a parallel positive Dehn twist will be called $K$.  When adding Dehn twists to the monodromy that are parallel to $K$, we change the manifold, but we still call that same boundary component $K$.  With this abuse of notation in mind, we are now ready to construct open books corresponding to transverse surgery.

\begin{prop}\label{prop:admissible OBD} Let $r < 0$ be a rational number, with $r = [a_1+1, a_2, \ldots, a_n]$.  The open book supporting admissible transverse $r$-surgery with respect to the page slope on the binding component $K$ is obtained by, for each $i = 1, \ldots, n$ in order, stabilising $K$ positively $|a_i+2|$ times and adding a positive Dehn twist about $K$.  \end{prop}
\begin{proof} We first show that this topologically supports the manifold we are interested in, and then we will show that it supports the contact structure coming from admissible transverse surgery.  We prove the topological statement by induction on $n$, showing that in addition, the page slope after surgery is $[a_1+1, a_2, \ldots, a_n+1]$.  If $n = 1$, then $r$ is a negative integer, and we can see that positively stabilising $K$, $|r+1|=|a_1+2|$ times starting from the standard Farey graph will move the label $r+1$ to East, and North's label will be $r$.  Then $(-1)$-surgery on the binding component $K$ in the new page will give us $r$ surgery on $K$, with respect to the original page slope.  This $(-1)$-surgery corresponds to adding a positive Dehn twist about the binding, by Theorem~\ref{thm:OBD connected binding}.  Note also that the page slope after this process is equal to $r+1$.

If $n > 1$, then let $r' = [a_1+1, a_2, \ldots, a_{n-1}]$, and let $r'' = [a_1+1,a_2,\ldots,a_{n-1}+1].$  Then $r' < r < r''$, and there is a geodesic in the Farey graph between $r'$ and $r''$, and between $r'$ and $r$, by Lemma~\ref{Farey sum} and Lemma~\ref{Farey connected}. Applying the induction hypothesis to $r'$, we can do $r'$ surgery using the construction, and in the process, send $r''$ to the page slope.  We claim that in this new coordinate system, $r$ gets sent to the negative integer $a_n+1$.  Indeed, note that since $r'$ gets sent to $\infty$ (the meridional slope), and $r''$ gets sent to $0$ (the page slope), $r$ gets sent to $|a_n+1|\cdot(\infty)\oplus (0/(-1))$, which is $a_n+1$.  Now we can lower the page slope to $a_n+2$ by doing $|a_n+2|$ stabilisations of $K$, and then do $(-1)$-surgery with respect to the new page slope, which corresponds to adding a negative Dehn twist around the boundary.  Thus the open book topologically supports the manifold coming from $r$-surgery on $K$, with respect to the page slope.

We now show that the supported contact structure is that coming from admissible transverse surgery.  The algorithm consists of sequences of stabilisations and surgeries on the knot $K$ (which we have tracked through this process, as in the discussion before this proposition).  Note that the transverse knot $K$ before any stabilisation is transverse isotopic to the binding component $K$ after stabilisation.  In addition, admissible surgery on a knot $K$, followed by admissible surgery on the dual knot to $K$, \textit{ie.\@} the core of the surgery torus, is equivalent to a single admissible surgery on $K$.  Thus, since stabilisation does not change the transverse knot type of $K$, the algorithm consists of a series of surgeries on successive dual knots.  It is enough to show that each of these surgeries corresponds to an admissible transverse surgery to show that the entire algorithm corresponds to an admissible transverse surgery.  But each individual surgery (that is, adding a boundary-parallel positive Dehn twist to the monodromy) is topologically a $(-1)$-surgery with respect to the page slope, and thus by Theorem~\ref{BEV OBD Thm} and Proposition~\ref{prop:plusminus one surgery}, the resulting open book supports the contact structure coming from admissible transverse surgery on the binding component.\end{proof}

\begin{remark} Compare our construction to Baker, Etnyre, and van Horn-Morris \cite{BEVHM}, who construct open books for the result of admissible transverse surgery by resolving the induced rational open book using cables of the binding.  Our construction consists of taking a cable at each step (\textit{ie.\@} for each $i$ in Proposition~\ref{prop:admissible OBD}).   Our construction agrees with theirs when the surgery coefficient is a negative integer. \end{remark}

\begin{example}\label{eg:neg8over5} We can calculate that $-8/5 = [-3+1,-3,-2]$.  Thus, given a knot $K$ in the binding of an open book, to perform admissible transverse $-8/5$-surgery on $K$ with respect to the page slope, we would stabilise once, add a positive Dehn twist around $K$, stabilise once, and then add two positive Dehn twists about $K$. To see this, we track the labels on West and East through this process, via an ordered pair $(\rm{West}, \rm{East}$).  We start with $(\rm{West}, \rm{East}) = (\infty,0)$, the standard labeling.  The initial stabilisation creates the labeling $(\infty,-1)$.  The $(-1)$-surgery (corresponding to adding the positive Dehn twist) changes the labeling to $(-2,-1)$.  Another stabilisation keeps West's label fixed and creates the labeling $(-2,-3/2)$.  The $(-1)$-surgery changes the labels to $(-5/3,-3/2)$.  At this point, North's label is $-8/5$.  Finally, the last $(-1)$-surgery creates the labeling $(-8/5,-3/2)$. See Figure~\ref{fig:8over11}, where if you ignore the negative Dehn twists, the open book on the left is admissible transverse $-8/5$-surgery on the maximum self-linking right-handed trefoil in $(S^3,\xi_{\rm{std}})$.\end{example}

\subsubsection{Open books compatible with inadmissible transverse surgery}
\label{sec:inadmissible OBD}

If $r > 0$ with respect to the page slope, pick the least positive integer $n$ such that $1/n \leq r$.  Then by Lemma~\ref{inadmissible then admissible}, doing inadmissible $1/n$ surgery followed by admissible $r'$-surgery, for some $r' < 0$, is equivalent to doing inadmissible $r$-surgery.  This corresponds to adding $D_K^{-n}$ to the monodromy before working out the open book for admissible $r'$-surgery, by Theorem~\ref{thm:OBD connected binding}.

To work out $r'$, note that if $1/n \leq r = p/q < 1/(n-1)$, then $$\frac{p}{q} = a\cdot\left(\frac{1}{n}\right) \oplus b\cdot\left(\frac{1}{n-1}\right) = \frac{a+b}{an+b(n-1)}$$ for some positive integers $a$ and $b$. This doesn't directly correspond to operations on the Farey graph itself, but is purely an algebraic assertion. We can write this as
$$\vect{q}{p}=\matrixb{n}{n-1}{1}{1}\vect{a}{b}.$$
The matrix is invertible, so we see that $a = q + p - np$, and $b = np - q$.

After doing inadmissible $1/n$-surgery, we note that on the Farey graph, the label $1/n$ has moved to West, the label $1/(n-1)$ has moved to North.  The label we have called $r'$ is at the point in between West and North corresponding to how $r$ was sitting relative to $1/n$ and $1/(n-1)$.  We can conclude that $r' = a\cdot(\infty)\oplus b\cdot(-1)$, or more explicitly,
$$r' = a\cdot\left(\frac{1}{0}\right)\oplus b\cdot\left(\frac{1}{-1}\right) = \frac{(q+p-np)(1)+(np-q)(1)}{(q+p-np)(0)+(np-q)(-1)} = \frac{p}{q-np}.$$

\begin{prop}\label{prop:inadmissible OBD} Let $r = p/q > 0$ be a rational number, and let $n$ be a positive integer such that $1/n < r$.  The open book supporting inadmissible transverse $r$-surgery with respect to the page slope on the binding component $K$ is obtained by first adding $n$ positive Dehn twists about $K$, and then performing transverse admissible $r'$-surgery on $K$ (as in Prop~\ref{prop:admissible OBD}), where $r' = \frac{p}{q-np}.$  \end{prop}

\begin{remark} In the case of integer surgery, our open books are identical to those of Lisca and Stipsicz \cite{LS:surgery}.  In addition, although he does not explicitly construct such examples, our open books for integer surgery can be constructed using the fibre sum operation of Klukas \cite{Klukas}.  The construction of Baker, Etnyre, and van Horn-Morris in \cite{BEVHM} does not cover the case $r > 0$.  \end{remark}

\begin{remark} Since the construction depends on the page slope, lowering the page slope by positive stabilisations prior to beginning the construction will give alternative open books for the same inadmissible surgery.  Note that by stabilising sufficiently many times, we can ensure that the page slope is smaller than any given slope.  Thus this construction allows us to create open books for any inadmissible transverse surgery, even if the slope is smaller than the original page slope.  Of course, if the inadmissible surgery coefficient is lower than the original page slope, the result is overtwisted.  \end{remark}

\begin{example} If we stabilise $3$ times, the effect on the Farey graph labeling is to keep West's label fixed, to change East's label to $-3$, and to change South's label to $-2$.  Thus, performing $(+1)$-surgery with respect to the new page slope relabels West with $-2$, and hence is the same as performing inadmissible transverse $-2$-surgery on the binding component.\end{example}

\begin{figure}[htbp]
\begin{center}
\begin{overpic}[scale=1.5]{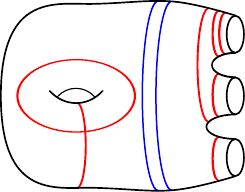}
\put(95,120){\LARGE $-$}
\put(120,120){\LARGE $-$}
\end{overpic}
\hspace{2cm}\begin{overpic}[scale=1.5]{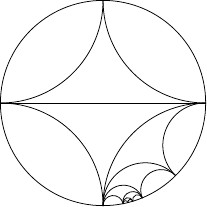}
\put(-12,71){$\infty$}
\put(151,71){$0$}
\put(65,151){$-1$}
\put(72,-8){$1$}
\put(126,14){$\frac{1}{2}$}
\put(103,-3){$\frac{2}{3}$}
\put(83,-9){$\frac{3}{4}$}
\put(96,-6){$\frac{5}{7}$}
\put(87,-8){$\frac{8}{11}$}
\end{overpic}
\caption{On the left, an open book for inadmissible transverse $8/11$-surgery on the maximum self-linking right-handed trefoil in $(S^3,\xi_{\rm{std}})$.  Positive Dehn twists are red and unmarked, and negative Dehn twists are blue and marked with a minus sign.  The Farey graph leading to $8/11$ is on the right.  Without the negative Dehn twists, this corresponds to admissible transverse $-8/5$-surgery on the same knot.}
\label{fig:8over11}
\end{center}
\end{figure}
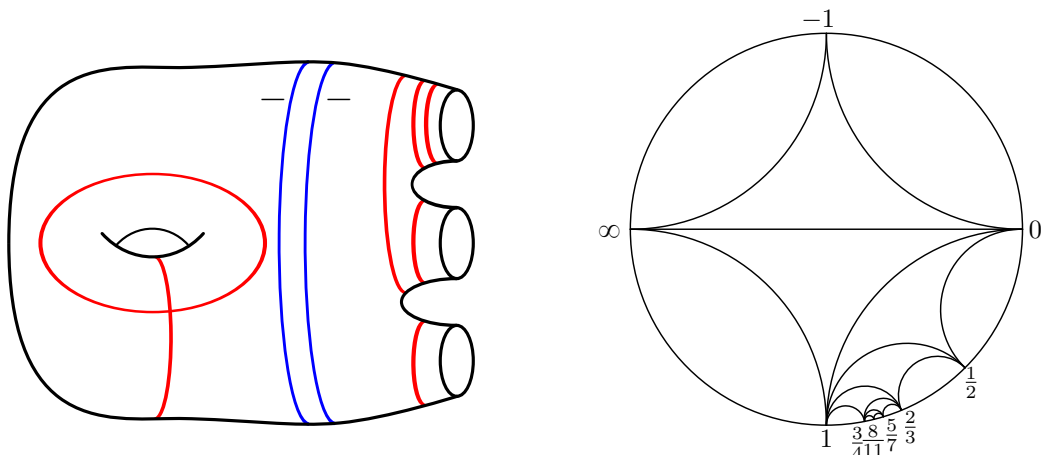

\begin{example} In Figure~\ref{fig:8over11}, we see an open book for inadmissible transverse $8/11$-surgery on the right-handed trefoil in $(S^3,\xi_{\rm{std}})$ with self-linking $1$.  We see this by noticing that $1/2 < 8/11 < 1$, so we start by performing two negative Dehn twists around the binding.  After that, we follow the process for admissible transverse $-8/5$-surgery, as $\frac{8}{11-2\cdot8} = -\frac{8}{5}$. This then follows the same steps as in Example~\ref{eg:neg8over5}. \end{example}

\begin{example} We can describe half and full Lutz twists using inadmissible transverse surgery.  A half Lutz twist is given by inadmissible transverse $\infty$-surgery, that is, adding a $\pi$ rotation of the contact planes onto $M\backslash K$.  A full Lutz twist is two half Lutz twists.  We can create an open book for $0$-surgery on a binding component by (a) stabilising, and then (b) doing a $(+1)$-surgery with respect to the page.  This gives a ``quarter" Lutz twist.  Doing this twice with a positive stabilisation in between will be a half Lutz twist, and four times with positive stabilisations separating each application will be a full Lutz twist. The positive stabilisation in the middle serves to set the new page slope to either $0$ or $\infty$, so that $0$-surgery (\textit{ie.\@} page slope surgery) will add another $\pi$ rotation to the contact structure

While doing a quarter Lutz twist, before doing step (a), we can do $1/n$ surgery on the binding component with respect to the page, for any integer $n \neq 0$, and then perform steps (a) and (b).  This gives a family of open books for a quarter Lutz twist all with the same page.  The Dehn twists involved in this $1/n$ surgery correspond to contact $(\pm1)$-surgery on boundaries of overtwisted discs, which do not change the manifold or the isotopy class of the contact structure.  Note, though, that if we want to compose quarter Lutz twists, instead of performing a positive stabilisation before repeating steps (a) and (b), the normalisation required to make the page slope equal to the original meridian slope will depend on $n$.  Thus, we can produce a family of open books for a half and full Lutz twist, many with the same page.  Ozbagci and Pamuk \cite{OP}, based on work of Ding, Geiges, and Stipsicz \cite{DGS:Lutz} and Etnyre \cite{Etnyre:planar}, have also obtained open books for a Lutz twist, based on a contact surgery diagram.  Our construction differs from theirs, and produces not easily comparable open books, as they do a Lutz twist on the transverse push-off of a Legendrian knot lying in the page of an open book, whereas our transverse knot is in the binding. \end{example}

\subsection{Comparing Transverse Surgery to Contact Surgery}

\begin{proof}[Proof of Theorem~\ref{Transverse=Legendrian}]
Baker, Etnyre, and van Horn-Morris \cite[Lemma 6.5]{BEVHM} discuss how to take any transverse knot $K \subset \Mxi$ and put it in the binding of an open book compatible with $\Mxi$.  If $K$ is the unique binding component, do a boundary-parallel positive stabilisation on the open book.  Now a push-off $L$ of the binding can be Legendrian realised on this new page, with contact framing number given by the new page slope.  Notice that if we stabilise the open book (again, if we had to stabilise at the beginning) at the boundary corresponding to $K$, a Legendrian approximation $L'$ realised on the page of the doubly stabilised open book is isotopic to a negative Legendrian stabilisation of the Legendrian knot $L$.  We can see that this stabilisation is negative by noticing that $L$ and $L'$ are Legendrian approximations of the same transverse knot: negative stabilisation preserves the isotopy class of the transverse push-off, whereas positive stabilisation does not.

Proposition~\ref{prop:plusminus one surgery} shows that adding to the monodromy a positive (resp.\ negative) Dehn twist on a push-off of the binding, after Legendrian realisation, is equivalent to contact $(-1)$-surgery (resp.\ contact $(+1)$-surgery). Thus the open book construction from Proposition~\ref{prop:admissible OBD}, is a sequence of contact $(\pm1)$-surgeries on a Legendrian approximation of $K$ and its negative stabilisations.

We check that the sequence of surgeries and stabilisations provided by Proposition~\ref{prop:inadmissible OBD} for $r$-surgery ($r>0$ with respect with to the page slope) is the same as that provided by Ding and Geiges \cite{DG:surgery} in Construction~\ref{DingGeiges} for contact $r$-surgery on $L$, with all negative bypass layers:

\begin{itemize}
\item the contact $(+1)$-surgeries performed at the beginning of Construction~\ref{DingGeiges} correspond by Theorem~\ref{thm:OBD multiple boundaries} to adding negative Dehn twists to the monodromy around $K$,
\item the surgery coefficient (with respect to the contact framing) with which we have to negative contact surgery in Construction~\ref{DingGeiges} is exactly the same surgery coefficient $r'$ (which is with respect to the page slope),
\item negatively stabilising a push-off (resp.\ performing contact $(-1)$-surgery on a push-off) in Construction~\ref{DingGeiges} corresponds to stabilising the open book (operation (1) from Section~\ref{sec:construction}) (resp.\ operation (2), adding a boundary-parallel positive Dehn twist) in Proposition~\ref{prop:admissible OBD},
\end{itemize}

Finally, we see that the surgeries on Legendrian knots in Construction~\ref{DingGeiges} and corresponding open book operations in Proposition~\ref{prop:inadmissible OBD} are performed in the same order, and so the equivalence follows.
\end{proof}

Having identified inadmissible transverse surgery and positive contact surgery (with the choices that give $\xi^-$), we can now easily see that Corollary~\ref{Legendrian Surgery on Stabilisation} follows from Theorem~\ref{Transverse=Legendrian}.

\begin{proof}[Proof of Corollary~\ref{Legendrian Surgery on Stabilisation}]
Given a Legendrian $L$ in $\Mxi$, let $L_-$ be a single negative stabilisation of $L$, and let $K$ be a positive transverse push-off of $L$.  Given any $r \neq 0$, Theorem~\ref{Transverse=Legendrian} implies that contact $r$-surgery on $L$ (with all negative stabilisation choices in Construction~\ref{DingGeiges}) is equivalent to transverse $f_L+r$-surgery on $K$ (where $f_L$ is the contact framing on $L$, and the surgery is admissible if $r < 0$ and inadmissible if $r > 0$).  Also according to Theorem~\ref{Transverse=Legendrian}, (in)admissible transverse $f_L+r$-surgery on $K$ is equivalent to contact $(f_L+r-f_{L_-})$-surgery on $L_-$, which is contact $(r+1)$-surgery on $L_-$ (with all negative stabilisation choices).
\end{proof}

\section{Tight Surgeries}
\label{sec:tight surgeries}

In this section, we will use results from Heegaard Floer homology to prove Theorem~\ref{thm:preserve tightness}.  Our results also allow us to say more about some contact geometric invariants, namely the contact width and the tight transverse surgery interval.  We comment on these at the end of this section.

\subsection{Heegaard Floer Homology}

We outline the relevant constructions involved in Heegaard Floer homology, as well as a few theorems.  See \cite{OS:hf1,OS:hf2} for more details.

Given a closed $3$-manifold $M$, we choose a Heegaard decomposition $(\Sigma, \bm{\alpha},\bm{\beta})$ of $M$.  Here, $\Sigma$ is a genus $g$ surface, $\bm{\alpha} = \{\alpha_1, \ldots, \alpha_g\}$ and $\bm{\beta} = \{\beta_1,\ldots,\beta_g\}$ are homologically-independent collections of essential simple closed curves on $\Sigma$ such that $\alpha_i \cap \alpha_j = \beta_i \cap \beta_j = \emptyset$ for $i \neq j$. Furthermore, we recover $M$ by attaching $2$-handles to $\Sigma \times [0,1]$ along the $\bm{\alpha} \times \{0\}$ and $\bm{\beta} \times \{1\}$ curves, and adding $3$-handles to the resulting boundary.

Given such a Heegaard decomposition, we choose a point $z \in \Sigma \backslash\left(\bm{\alpha}\cup\bm{\beta}\right).$  The \textit{Heegaard Floer chain group} $\cf(\Sigma,\bm{\alpha},\bm{\beta},z)$ is generated by $g$-tuples of intersections of $\bm{\alpha}$ with $\bm{\beta}$, such that there is one intersection point on each of the $\bm{\alpha}$ and each of the $\bm{\beta}$ curves.  The differential is provided by counting holomorphic curves in the $g$-fold symmetric product of $M$. The homology of this complex is independent of all choices.  It is called the \textit{Heegaard Floer homology} of $M$, and is written $\hf(M)$.

Given a cobordism $W:M \to N$, where $\bd W = -M \cup N$, it induces a map $F_W:\hf(M) \to \hf(N)$.  This map can be seen by counting certain holomorphic triangles in a symmetric product of the Heegaard surface $(\Sigma,\bm{\alpha},\bm{\beta},\bm{\gamma},z)$.  Here, $(\Sigma,\bm{\alpha},\bm{\beta},z)$ defines $M$, and $(\Sigma,\bm{\alpha},\bm{\gamma},z)$ defines $N$.  The map $F_W$ induced by the cobordism satisfies the following adjunction inequality.  (The actual result is more refined than we present it here, but this is all we will need.)

\begin{thm}[Ozsváth--Szabó \cite{OS:4manifolds1}]\label{hfadjunction} If $W$ contains a homologically non-trivial closed surface $S$ with genus $g(S) \geq 1$ such that $[S] \cdot [S] > 2g(S) - 2$, then $F_W = 0$. \end{thm}

Given an open book decomposition $\Sigphi$ for $\Mxi$, we can define a Heegaard decomposition of $-M$ and an element $c(\xi) \in \hf(-M)$ distinguished up to sign.  This was originally defined by Ozsváth and Szabó \cite{OS:contact}, but our presentation will follow the description of Honda, Kazez, and Matić \cite{HKM:contactclass}.  Let $\Sigma' = \Sigma_0 \cup -\Sigma_{1/2}$ be the Heegaard surface, where $\Sigma_t = \Sigma \times \{t\} \subset \Sigma \times [0,1]$ in the mapping torus construction of $M$ from the open book.  Choose a basis of arcs $\gamma_1, \ldots, \gamma_k$ for $\Sigma_0$, and for each $i$, let $\gamma'_i$ be a push-off of $\gamma_i$, where the endpoints are pushed in the direction of the orientation on the boundary of $\Sigma$.  We let the $\bm{\alpha}$ curves be $\alpha_i = \gamma_i \cup \gamma_i$, where the first $\gamma_i$ is sitting on $\Sigma_0$, and the second on $-\Sigma_{1/2}$.  We let $\beta_i = \gamma'_i \cup \phi^{-1}(\gamma'_i)$, where again the first $\gamma'_i$ is sitting on $\Sigma_0$, and $\phi^{-1}(\gamma'_i)$ is sitting on $-\Sigma_{1/2}$.  Place $z$ in $\Sigma_0$.  Then $(\Sigma, \bm{\beta}, \bm{\alpha}, z)$ is a Heegaard diagram for $-M$, where we switch the roles of $\bm{\alpha}$ and $\bm{\beta}$ to get the correct orientation.

For each $i$, the curves $\alpha_i$ and $\beta_i$ intersect each other exactly once inside $\Sigma_0$, at the point $c_i$.  It can be shown that the generator $\bm{c} = \{c_1,\ldots,c_k\}$ is a cycle, and thus it defines a class $\bm{c} \in \hf(-M)$.  Honda, Kazez, and Matić \cite{HKM:contactclass} identify this class with a class previously defined by Ozsváth and Szabó \cite{OS:contact}.  Ozsváth and Szabó show that this class is independent of the choice of open book decomposition for $\Mxi$, and defines the \textit{Heegaard Floer contact invariant} $c(\xi)$ of $\Mxi$.

\begin{thm}[Ozsváth--Szabó \cite{OS:contact}] The Heegaard Floer contact invariant $c(\xi)$ of $\Mxi$ satisfies the following properties.
\begin{itemize}
\item If $\Mxi$ is overtwisted, then $c(\xi) = 0$.
\item If $\Mxi$ is strongly or Stein fillable, then $c(\xi) \neq 0$.
\item If $W:\Mxi \to (N,\xi')$ is the cobordism induced by contact $(+1)$-surgery on a Legendrian knot, then $F_{-W}(c(\xi)) = c(\xi').$
\end{itemize}
\end{thm}

Given an open book $\Sigphi$ supporting $\Mxi$, and a component of $\bd\Sigma$, we can create a new open book $(\Sigma',\phi')$ supporting $(M', \xi')$ by \textit{capping off} the binding component, \textit{ie.\@} by letting $\Sigma'$ be $\Sigma$ union a disc glued along the binding component, and where $\phi'$ is an extension of $\phi$ by the identity over the disc (see \cite{Baldwin:cappingoff}).  If $(\Sigma,\phi \circ D^{-1}_\bd)$ supports $(M'',\xi'')$, where $D^{-1}_\bd$ is a negative Dehn twist around the boundary component that was capped off, then Ozsváth and Szabó have proved the following.

\begin{thm}[Ozsváth--Szabó \cite{OS:hf2}]
\label{thm:exact triangle}
The following sequence is exact.
\begin{center}
\begin{tikzpicture}[node distance=2cm,auto]
  \node (L) {$\hf(-M')$};
  \node (R) [right of=L, node distance=4cm] {$\hf(-M)$};
  \node (B) [below of=L, right of=L, node distance=2cm] {$\hf(-M'').$};
  \draw[->] (L) to (R);
  \draw[->] (R) to (B);
  \draw[->] (B) to (L);
\end{tikzpicture}
\end{center}
\end{thm}

\subsection{Heegaard Floer contact invariant of surgeries}

Hedden and Plamenevskaya \cite{HP} track the non-vanishing of the Heegaard Floer contact invariant after surgery on the binding component of an open book.  In light of Theorem~\ref{OBDThm}, we can restate their result in terms of inadmissible transverse surgery.

\begin{thm}[Hedden--Plamenevskaya \cite{HP}] \label{thm:HP} If $K$ is an integrally fibred knot in a closed 3-manifold $M$, and the contact structure $\xi$ supported by the open book with binding $K$ is such that $c(\xi) \neq 0$, then $c(\xi_r) \neq 0$ for all $r \geq 2g$, where $\xi_r$ is the contact structure on $M_r(K)$ coming from inadmissible transverse $r$-surgery on $K$. \end{thm}

We will extend that result to prove that $c(\xi_r) \neq 0$ for all $r > 2g-1$, and then use this to prove Theorem~\ref{thm:preserve tightness}.  To do this, we will use our open books for inadmissible surgery created in Section~\ref{sec:OBD} along with the following theorem.  The line about tightness follows from a recent paper of Wand \cite{Wand}.

\begin{thm}[Baker--Etnyre--van Horn-Morris \cite{BEVHM}, Baldwin \cite{Baldwin:monoid}]\label{thm:monoid} Fix a surface $\Sigma$ with boundary.  The set of monodromies $\phi \in \rm{Diff}^+(\Sigma,\bd\Sigma)$ defining contact structures with a fixed property from the following list form a monoid in $\rm{Diff}^+(\Sigma,\bd\Sigma)$:
\begin{itemize}
\item tight,
\item non-vanishing Heegaard Floer invariant,
\item weakly fillable,
\item strongly fillable,
\item Stein fillable.
\end{itemize}
\end{thm}

The proof of Theorem~\ref{thm:preserve tightness} will be as follows: we will consider model open books relevant to $r$-surgery on the connected binding of a genus $g$ open book; we will show that for $r > 2g-1$, these open books have non-vanishing Heegaard Floer contact invariant, and in particular, are tight; finally, Theorem~\ref{thm:monoid} will allow us to prove the result for a generic monodromy that supports a tight contact structure.  The proof of the non-vanishing of the Heegaard Floer contact invariants of the model open books is based on proofs of Lisca and Stipsicz \cite{LS:hf1}.

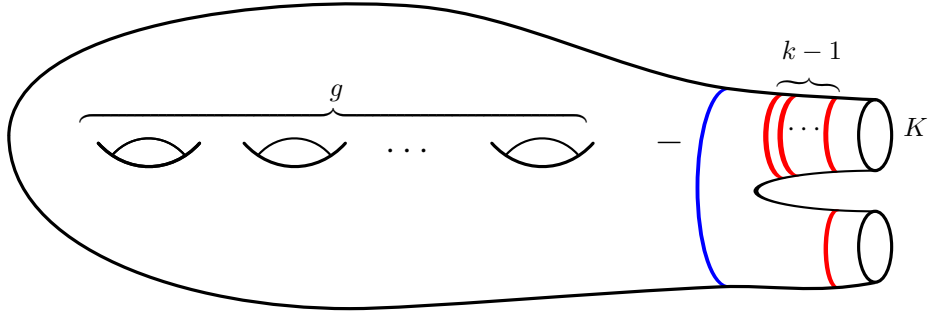
\begin{figure}[htbp]
\begin{center}
\begin{overpic}[scale=1.5,tics=15]{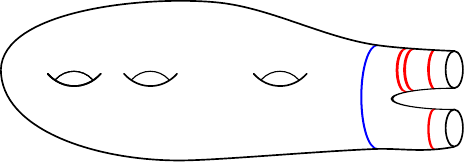}
\put(245,60){\LARGE $-$}
\put(143,57){\LARGE $\cdots$}
\put(122,81){$g$}
\put(26,58){\rotatebox{90}{$\left.\begin{array}{c} \vspace{6cm} \\ ~ \end{array}\right\}$}}
\put(293,95){$k-1$}
\put(290,71){\rotatebox{85}{$\left.\begin{array}{c} \vspace{0.3cm} \end{array}\right\}$}}
\put(295,66){\large $\cdots$}
\put(339,66){$K$}
\end{overpic}
\caption{For an integer $k \geq 1$, this is the open book supporting inadmissible transverse $k$-surgery on the binding $K$ of the genus $g$ open book with connected binding and trivial monodromy.  The unmarked red curves are positive Dehn twists and the blue curve marked with a negative sign is a negative Dehn twist.  The upper boundary component represents the dual knot to $K$, \textit{ie.\@} the core of the surgery torus.  We also call this knot $K$.}
\label{fig:sigmagn after negative Dehn twist and destab}
\end{center}
\end{figure}

Let $(M_g,\xi_g)$ be the contact manifold supported by the open book $(\Sigma_g^1,\rm{id}_{\Sigma_g^1})$, where the page has genus $g$ and a single boundary component.  The manifold $M_g$ is a connect sum of $2g$ copies of $S^1 \times S^2$, and $\xi_g$ is the unique tight contact structure on $M_g$.  Note that $c(\xi_g) \neq 0 \in \hf(-M_g)$.  Let the binding of this open book be $K$.

Consider the open book in Figure~\ref{fig:sigmagn after negative Dehn twist and destab} (with $k = 2g$) for inadmissible transverse $2g$-surgery on $K$.  From Theorem~\ref{thm:HP}, we can conclude that these support tight contact structures with non-vanishing Heegaard Floer invariant.  We will use these open books, and the Heegaard Floer exact triangle, to show that the following model open books $(\Sigma_{g,n}, \phi_{g,n})$ in Figure~\ref{fig:sigmagn} have non-vanishing Heegaard Floer invariant.

\begin{lem}\label{lem:OBD is right surgery} The open book $(\Sigma_{g,n},\phi_{g,n})$ in Figure~\ref{fig:sigmagn} supports the contact structure coming from inadmissible transverse $(2g-1+1/n)$-surgery on the binding of the connected binding genus $g$ open book with trivial monodromy. \end{lem}
\begin{proof}
We start with the open book $(\Sigma_g^1, \rm{id}_{\Sigma_g^1})$, where $\Sigma_g^1$ is a once-punctured genus $g$ surface.  The binding $K$ has page slope $0$, as the page represents a Seifert surface for the binding.  After doing $1$-surgery with respect to the page slope, we have done inadmissible transverse $1$-surgery, and the in the Farey graph, West is labeled $1$, while East is still labeled $0$.  Stabilising $K$ once labels the Farey graph with $1$ on West, $\infty$ on East, and $+2$ on North.  If $g > 1$, then we do $-1/(2g-2)$-surgery with respect to the page slope, that is, we add $2g-2$ boundary parallel positive Dehn twists to $K$; West is now labeled $2g-1$, leaving $\infty$ at East, and $2g$ at North.  Note that
$$n\cdot\frac{2g-1}{1}\oplus\frac{2g}{1} = \frac{(2g-1)n+2g-1+1}{n+1} = 2g-1 + \frac{1}{n+1}.$$  Thus doing $n-1$ positive stabilisations brings the label $2g-1+1/n$ to North, and so adding a positive Dehn twist around $K$ will bring the label $2g-1+1/n$ to West.  Since West is the meridian slope, this open book is inadmissible transverse $(2g-1+1/n)$-surgery on $K$.
\end{proof}

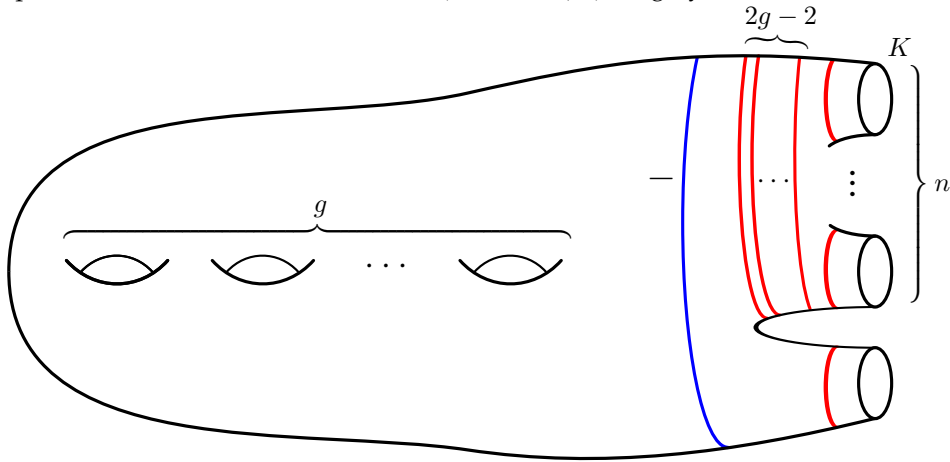
\begin{figure}[htbp]
\begin{center}
\begin{overpic}[scale=1.5,tics=15]{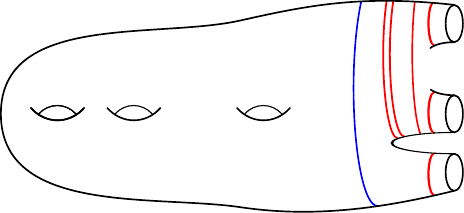}
\put(242,103){\LARGE $-$}
\put(135,70){\LARGE $\cdots$}
\put(329,102){$\left.\begin{array}{c} \vspace{2.45cm} \\ ~ \end{array}\right\} n$}
\put(116,94){$g$}
\put(20,71){\rotatebox{90}{$\left.\begin{array}{c} \vspace{6cm} \\ ~ \end{array}\right\}$}}
\put(279,165){$2g-2$}
\put(278,144){\rotatebox{85}{$\left.\begin{array}{c} \vspace{0.3cm} \end{array}\right\}$}}
\put(283.5,103){\large $\cdots$}
\put(333,153){$K$}
\put(318,100){\LARGE $\vdots$}
\end{overpic}
\caption{The open book $(\Sigma_{g,n},\phi_{g,n})$, where the unmarked red curves are positive Dehn twists and the blue curve marked with a negative sign is a negative Dehn twist.  The upper boundary component represents the original binding $K$.}
\label{fig:sigmagn}
\end{center}
\end{figure}

Let $(M_{g,n},\xi_{g,n})$ denote the contact manifold obtained from the open book $(\Sigma_{g,n},\phi_{g,n})$ in Lemma~\ref{lem:OBD is right surgery}.  Let $(M_{g,\infty},\xi_{g,\infty})$ denote the manifold obtained by adding a negative Dehn twist along $K$ in $(\Sigma_{g,n},\phi_{g,n})$.  Note that this destabilises, and hence is independent of $n$.  The supported manifold is inadmissible transverse $(2g-1)$-surgery on $K$.  See Figure~\ref{fig:sigmagn after negative Dehn twist and destab} for $k = 2g-1$.

By capping off the boundary component of $(\Sigma_{g,n+1},\phi_{g,n+1})$ immediately under $K$ in Figure~\ref{fig:sigmagn} when $n \geq 1$, Theorem~\ref{thm:exact triangle} gives us the exact sequence

\begin{center}
\begin{tikzpicture}[node distance=2cm,auto]
  \node (L) {$\hf(-M_{g,n})$};
  \node (R) [right of=L, node distance=4cm] {$\hf(-M_{g,n+1})$};
  \node (B) [below of=L, right of=L, node distance=2cm] {$\hf(-M_{g,\infty}).$};
  \draw[->] (L) to node {$F_{-X_n}$} (R);
  \draw[->] (R) to (B);
  \draw[->] (B) to node {$F_{-Y_n}$} (L);
\end{tikzpicture}
\end{center}

Here, $F_{-X_n}$ and $F_{-Y_n}$ are the maps induced by reversing the orientation on the $4$-manifold cobordisms $X_n$ and $Y_n$, which are between $M_{g,n}$ and $M_{g,n+1}$, and $M_{g,n+1}$ and $M_{g,\infty}$, respectively.

\begin{lem}\label{lem:cobordism vanishing} $F_{-Y_n}:\hf(-M_{g,\infty}) \to \hf(-M_{g,n})$ is the $0$ map. \end{lem}
\begin{proof}
Recall that $M_{g,\infty}$ is the result of inadmissible tranvserse $2g-1$-surgery on the knot $K \subset M_g$ which is the binding of the open book $(\Sigma_g^1,\rm{id}_{\Sigma_g^1}).$  Let $K' \subset M_{g,\infty}$ denote the knot surgery dual to $K$, \textit{ie.\@} the core of the surgery torus.   Now, $Y_n$ is the cobordism from $M_{g,\infty}$ to $M_{g,n}$ given by some surgery on $K'$.

Since $M_{g,\infty}$ is topologically $(2g-1)$-surgery on $K$, we see that $K'$ is a rationally null-homologous knot with a rational Seifert surface $\Sigma$ of genus $g$.  In order to figure out what surgery on $K'$ will give us $M_{g,n}$, we look at the labeling of the Farey graph that results from doing inadmissible transverse $2g-1$-surgery on $K$ to get $M_{g,\infty}$.  By following the algorithm described in Section~\ref{sec:inadmissible OBD}, we see that after doing $2g-1$ surgery on the binding component $K$, West is labeled $2g-1$, East is labeled $\infty$, and the point that gets labeled $0$ is the point that on the standard labeling gets the label $1/(2g-1)$.  The point that is currently labeled $2g-1+1/n$ would get labeled $-n$ in the standard labeling.  Thus, $(-n)$-surgery on $K'$ with respect to the page slope will set the meridian to $2g-1+1/n$, which corresponds to inadmissible transverse $(2g-1+1/n)$-surgery on the original knot $K$, \textit{ie.\@} will give us $M_{g,n}$.

Since the label $0$ corresponds to the slope of the Seifert surface $\Sigma$ for $K$ in $M_g$, the rational Seifert surface for $K'$ in $M_{g,\infty}$ has slope $1/(2g-1)$ with respect to the page slope of the open book for $M_{g,\infty}$ in Figure~\ref{fig:sigmagn after negative Dehn twist and destab}, \textit{ie.\@} its intersection with the boundary $T$ of a neighourhood of $K'$ is a $(2g-1,1)$ curve, with respect to the $(\lambda,\mu)$ coordinates, where $\lambda$ comes from the page of the open book $(\Sigma_{g,\infty},\phi_{g,\infty}).$  Note that the open book in Figure~\ref{fig:sigmagn after negative Dehn twist and destab} is an integral open book, and so the page slope indeed defines a framing for $K'$.

Let $K'' \subset T$ be the $(2g-1,0)$ cable of $K'$ (with respect to the same framing), that is, a link of $2g-1$ copies of of the framing $\lambda$, which is homologically the same as the boundary of the rational Seifert surface $\Sigma$ minus a meridian.  To find a Seifert surface for $K''$, we take a meridional disc for $K'$ with boundary on $T$, and we think of the surface $\Sigma$ as also having boundary on $T$.  We create a new surface $\Sigma'$ by resolving all the intersections of the meridional disc and $\Sigma$ using negative bands.  The number of bands corresponds to the absolute value of the intersection number of $[\bd \Sigma]$ with $[\bd\Sigma']$.  Notice that since $\bd\Sigma' = (2g-1)[\lambda]$ (where $\lambda$ is isotopic to $K'$), the surface $\Sigma'$ has $2g-1$ boundary components.  Since we are adding one $0$-handle (the compressing disc) and $2g-1$, $1$-handles (the bands), we see that
$$\chi(\Sigma') = \chi(\Sigma) + 1 - (2g-1) = (1-2g)+1-2g+1 = 3-4g = 2 - 2g - (2g-1).$$ Thus $\Sigma'$ is a genus $g$ surface with $2g-1$ boundary components.

In the cobordism $Y_n$, we take a collar neighbourhood $M_{g,\infty} \times [0,1]$, and let $\Sigma'$ be a surface in the collar neighbourhood, with boundary in $M_{g,\infty} \times \{1\}$.  We build the cobordism $Y_n$ by attaching a $2$-handle to this collar neighbourhood.  Take $2g-1$ copies of the core of the $2$-handle with boundary on $M_{g,\infty} \times \{1\}$, and attach them to the $2g-1$ copies of $K'$ in the boundary of $\Sigma'$.  We will glue these to get a closed genus $g$ surface $\Sigma''$ with transverse double points.  We claim that there are $n\cdot{2g-1 \choose 2}$ double points.  Indeed, note that each core intersects each other core $n$ times (actually $-n$, but only the geometric intersection number is relevant here), and that there are ${2g-1\choose 2}$ pairs of cores to consider.  Thus, resolving these double points gives a surface $\widetilde{\Sigma}''$ of genus $g + n\cdot{2g-1\choose2}$.

To calculate the self-intersection of $\widetilde{\Sigma}''$, we note that since resolving the double points is a homologically trivial operation, and self-intersection depends only on homology class, we can calculate the self-intersection of $\Sigma''$ before resolving the double points.  We take a pushoff of $\Sigma''$ to calculate the self-intersection.  A push-off of each copy of the core of the $2$-handle intersects each of the $2g-1$ copies of the core of the $2$-handle in $\Sigma''$ a total of $-n$ times.  There are $(2g-1)^2$ pairs of cores to consider, so $$[\widetilde{\Sigma}'']\cdot[\widetilde{\Sigma}''] = [\Sigma'']\cdot[\Sigma''] = -n\cdot(2g-1)^2.$$

This is the self-intersection in $Y_n$, so in $-Y_n$, the self-intersection of $\widetilde{\Sigma}''$ is $n\cdot(2g-1)^2$.  Then we calculate in $-Y_n$:
$$[\widetilde{\Sigma}'']\cdot[\widetilde{\Sigma}'']-(2g(\widetilde{\Sigma}'')-2) = \left[n\cdot(2g-1)^2\right]-\left[2g+2n\cdot{2g-1\choose2}-2\right] = (2g-1)n-(2g-2).$$  This is $1$ for $n = 1$, and is increasing in $n$ for $g \geq 1$.  Thus, this is positive for all $g,n \geq 1$.

Since $-Y_n$ contains a surface $\widetilde{\Sigma''}$ with self-intersection greater than $2g(\widetilde{\Sigma''})-2$, the adjunction inequality of Theorem~\ref{hfadjunction} implies that $F_{-Y_n}$ is the $0$ map.
\end{proof}

\begin{lem}\label{lem:model OBD HF} The contact invariant $c(\xi_{g,n}) \neq 0$ for all $g,n \geq 1$. \end{lem}
\begin{proof}
Consider the Legendrian knot $L$ in Figure~\ref{fig:legendrian} with $tb(L) = 2g-1$ (\textit{cf.\@} \cite[Figure 1]{LS:non-fillable}). Note that $L$ is topologically isotopic to $K \subset (M_g,\xi_g)$, the binding of $(\Sigma_g^1, \rm{id}_{\Sigma_g^1})$.  Etnyre and Van Horn-Morris \cite[Theorem 1.6]{EVHM:fibered} show that if $T$ is a positive transverse push-off of $L$, then there is a contactomorphism of $(M_g,\xi_g)$, the contact manifold supported by $(\Sigma_g^1, \rm{id}_{\Sigma_g^1})$, such that $T$ is identified with $K$.  Thus, inadmissible transverse surgery on $K$ is contactomorphic to inadmissible transverse surgery on $T$, which by Theorem~\ref{Transverse=Legendrian} is identified with contact surgery on $L$.  By the proof of the invariance of the contact invariant, its non-vanishing is preserved under contactomorphism.  Hence, to show that $c(\xi_{g,n}) \neq 0$, it is enough to show that contact $(+1/n)$-surgery on $L$ has non-vanishing contact invariant.

\begin{figure}[htbp]
\begin{center}
\begin{overpic}[scale=.7,tics=30]{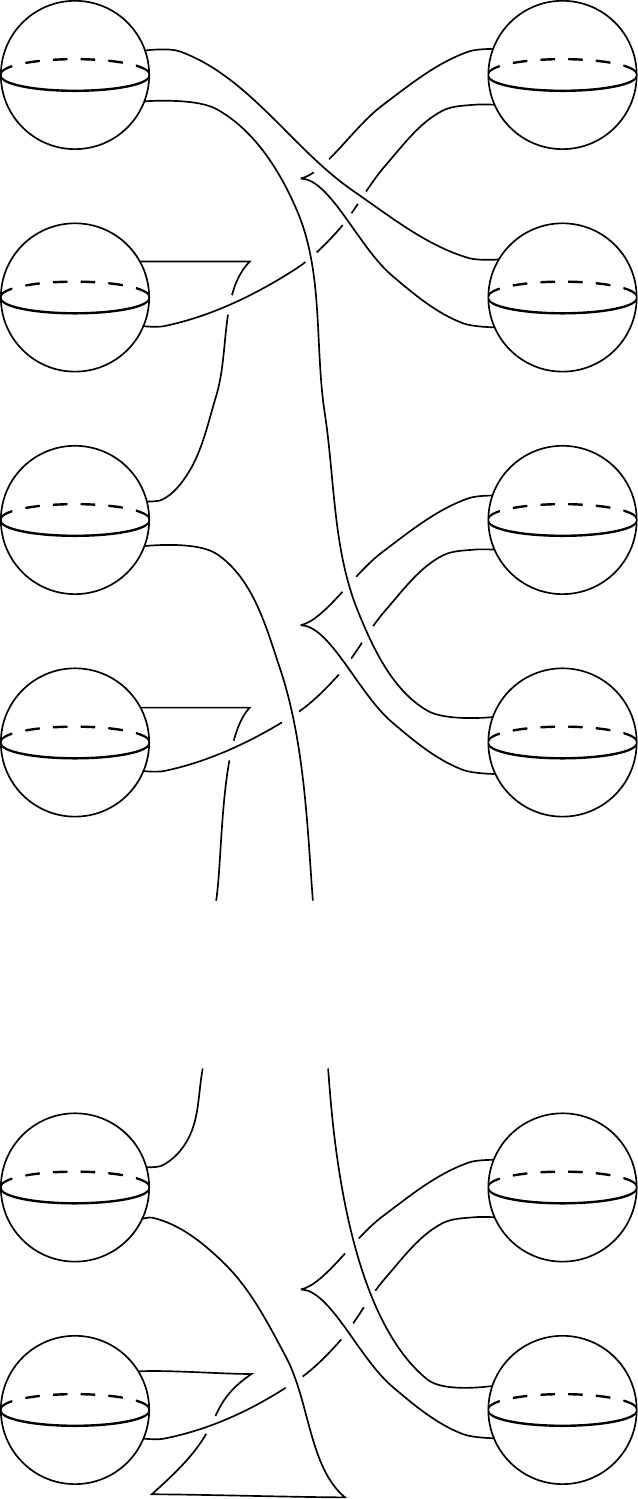}
\put(106,172){\LARGE $\vdots$}
\put(70,172){\LARGE $\vdots$}
\end{overpic}
\caption{The Legendrian knot $L$ with $tb(L) = 2g-1$ in $(M_g,\xi_g)$.  Each sphere on the left is identified with the corresponding sphere on the right, and there are $2g$ such pairs.  This image is a reproduction of \cite[Figure 1]{LS:non-fillable}.}
\label{fig:legendrian}
\end{center}
\end{figure}

Contact $(+1/n)$-surgery on $L$ can be realised as contact $(+1)$-surgeries on $n$ push-offs of $L$.  In addition, although $T$ might not be actually transverse isotopic to $K$, they are topologically isotopic, so the exact triangle of Heegaard Floer groups using the cobordisms $-X_n$ and $-Y_n$ apply to surgery on $T$ as well, and hence to surgeries on $L$.  We can thus realise $F_{-X_n}$ from $\hf(-M_{g,n})$ to $\hf(-M_{g,n+1})$ as being induced by contact $(+1)$-surgery, and thus $F_{-X_n}(c(\xi^-_{1/n}(L))) = c(\xi^-_{1/(n+1)}(L)).$   By Lemma~\ref{lem:cobordism vanishing}, $F_{-X_n}$ is an injective map.  By Theorem~\ref{thm:HP}, $c(\xi^-_{1}(L)) \neq 0$, as $(M_{g,1},\xi^-_{1}(L))$ is contactomorphic to $(M_{g,1},\xi_{g,1})$, and that is inadmissible transverse $2g$-surgery on $K$.  The lemma follows by induction on $n$.

Note that with a little extra effort, we can re-prove the non-vanishing of $c(\xi_{g,1})$, by showing that $F_{-X_0}$ is injective, where $F_{-X_0}$ is as in the exact triangle:

\begin{center}
\begin{tikzpicture}[node distance=2cm,auto]
  \node (L) {$\hf(-M_g)$};
  \node (R) [right of=L, node distance=4cm] {$\hf(-M_{g,1})$};
  \node (B) [below of=L, right of=L, node distance=2cm] {$\hf(-M_{g,\infty}).$};
  \draw[->] (L) to node {$F_{-X_0}$} (R);
  \draw[->] (R) to (B);
  \draw[->] (B) to node {$F_{-Y_0}$} (L);
\end{tikzpicture}
\end{center}
For this, we consider the cobordism $-Y_0$ backwards, as a cobordism from $M_g$ to $M_{g,\infty}$.  This is given by attaching a $2$-handle to $K$ with framing $2g-1$.  Thus, by gluing the core of the $2$-handle to a Seifert surface for $K$, we create a surface $\Sigma$ with genus $g$ and self-intersection $2g-1$.  By Theorem~\ref{hfadjunction}, $F_{-Y_0}$ is identically $0$, and hence $F_{-X_0}$ is injective.
\end{proof}

\begin{remark}
Alternatively, we can directly prove Lemma~\ref{lem:model OBD HF} without using the Legendrian knot $L$.  We pick a an open book for the surgery in order to make the Heegaard diagram simple. We then consider the Heegaard triple induced by capping off.  Using results from \cite{Baldwin:cappingoff}, we can narrow down the image $F_{-X_n}(c(\xi_{g,n}))$ to a sum of two classes, one of which represents $c(\xi_{g,n+1})$.  By considering $\Spinc$ structures and the conjugation map $\hf(-M_{g,n+1})$, we show that the class not representing $c(\xi_{g,n+1})$ is the zero class in homology, and the result follows.  See \cite{conwaythesis} for more details.
\end{remark}

\begin{proof}[Proof of Theorem~\ref{thm:preserve tightness}]
It is sufficient to prove that inadmissible transverse $(2g-1+1/n)$-surgery is tight for all $n \geq 1$, as given $r > 2g-1$, there is some $n$ such that $r > 2g-1+1/n$, and then inadmissible transverse $r$-surgery on $K$ is inadmissible transverse $(2g-1+1/n)$-surgery on $K$ followed by negative contact surgery on some link, and negative contact surgery preserves tightness, by Wand \cite{Wand}.  For the statement about the Heegaard Floer contact invariant, we note that Ozsváth and Szábo \cite{OS:contact} proved that negative contact surgery preserves non-vanishing of the contact invariant.

Let $\Sigphi$ be the genus $g$ open book for $\Mxi$ with connected binding $K$, and assume that $\xi$ is tight (resp.\ $c(\xi) \neq 0$).  If we plumb $\Sigma$ with an annulus, and extend $\phi$ over this annulus by the identity, we have an open book for $\left(M \varhash \left(S^1 \times S^2\right), \xi'\right)$, where $\xi' = \xi \varhash \xi_{\rm{std}}$, and $\xi_{\rm{std}}$ is the unique tight contact structure on $S^1 \times S^2$.  Since $\xi_{\rm{std}}$ is tight (resp.\ $c(\xi_{\rm{std}}) \neq 0$), we see that $\xi'$ is tight (resp.\ $c(\xi') \neq 0$).  So we can plumb enough copies of $(S^1 \times [0,1],\rm{id}_{S^1\times[0,1]})$ onto $\Sigphi$ such that our new surface $\Sigma'$ is homeomorphic to $\Sigma_{g,n}$, our new monodromy is $\phi'$, and our new supported contact structure $\xi'$ is tight (resp.\ has non-vanishing contact invariant).

Since both $\xi'$ and $\xi_{g,n}$ are tight (resp.\ $c(\xi') \neq 0$ and $c(\xi_{g,n}) \neq 0$), Theorem~\ref{thm:monoid} says that the open book $(\Sigma_{g,n},\phi'\circ\phi_{g,n})$ supports a tight contact structure (resp.\ with non-vanishing contact invariant).  But this is exactly the open book supporting inadmissible transverse $(2g-1+1/n)$-surgery on $K$.
\end{proof}

\subsection{Inadmissible Surgery on Links}
\label{subsec:link surgery}

Let $\Sigphi$ be an open book with multiple binding components supporting a tight contact structure. Based on the Theorem~\ref{thm:preserve tightness}, it is natural to ask whether sufficiently large surgeries on all the binding components of $\Sigphi$ would result in a tight manifold.  We show that this is in general not the case.  We will show that the model open books for surgery on multiple binding components are overtwisted.

To see this, we first define our model open books.  We start with $(\Sigma_g^n,\rm{id}_{\Sigma_g^n})$, where $\Sigma_g^n$ is a surface of genus $g$ with $n$ boundary components $K_1, \ldots, K_n$.  This is an open book for the connect sum of $2g+n-1$ copies of $S^1 \times S^2$, supporting the unique tight contact structure $\xi$ that is Stein fillable, and $c(\xi) \neq 0$.  Given rational numbers $r_1, \ldots, r_n$, let $(\Sigma_g^n(\bm{r}),\phi(\bm{r}))$ be the open book supporting the result of inadmissible transverse $r_i$-surgery on $K_i$, for each $i = 1,\ldots, n$.

\begin{thm}
\label{thm:link surgery}
The contact structure supported by $(\Sigma_g^n(\bm{r}),\phi(\bm{r}))$ for $n \geq 2$ is overtwisted for every $\bm{r} = (r_1, \ldots, r_n)$.
\end{thm}
\begin{proof}
If $r_i \leq 0$ for some $i$, then the contact manifold is already overtwisted.  Indeed, in the contact manifold supported by $(\Sigma_g^n,\rm{id}_{\Sigma_g^n})$, there is a Legendrian knot $L$ in a neighbourhood of $K_i$ with contact framing invariant equal to $r_i$ (in fact, it is topologically a cable of $K_i$): we can find a Legendrian representative of $K_i$ with contact framing equal to the page slope by realising a push-off of $K_i$ on the page, and this implies (by looking at a standard neighbourhood of $K_i$) that any contact framing less than this can also be realised in the knot type of some cable of $K_i$.  Thus, since after performing transverse inadmissible $r_i$-surgery on $K_i$ (with respect to the page slope), $L$ will bound an overtwisted disc, the contact manifold supported by $(\Sigma_g^n(\bm{k}),\phi(\bm{k}))$ is overtwisted.

So assume that $r_i > 0$ for each $i$.  Given $\bm{r}$, we define $\bm{r}'$ as follows.  Fix two distinct indices $i \neq j$, and let $r_i' = r_i - 1$, $r_j' = r_j + 1$, and let $r_l' = r_l$ for all $l \neq i, j$.  We claim that $(\Sigma_g^n(\bm{r}'),\phi(\bm{r}'))$ supports the same contact manifold as $(\Sigma_g^n(\bm{r}),\phi(\bm{r}))$.  The proof is contained in Figure~\ref{fig:link surgery} in the case that $\bm{r}$ is a collection of integers (although the proof is the same, the open book is more complicated when $r_i$ is non-integral).  Note that inadmissible transverse $r_j$-surgery on $K_j$ can be written in multiple ways, either directly (as in Figure~\ref{fig:sigmagn after negative Dehn twist and destab}) or by first stabilising $K_j$, and performing $(r_j+1)$-surgery with respect to the new page slope.  The boundary component created by stabilisation can now be thought of as coming from the $r_i$-surgery on $K_i$, thus changing the surgery coefficients as claimed.  We can decrease $k_i$ in this way until it is less than or equal to $0$, and then we are done, as in the first paragraph.
\end{proof}

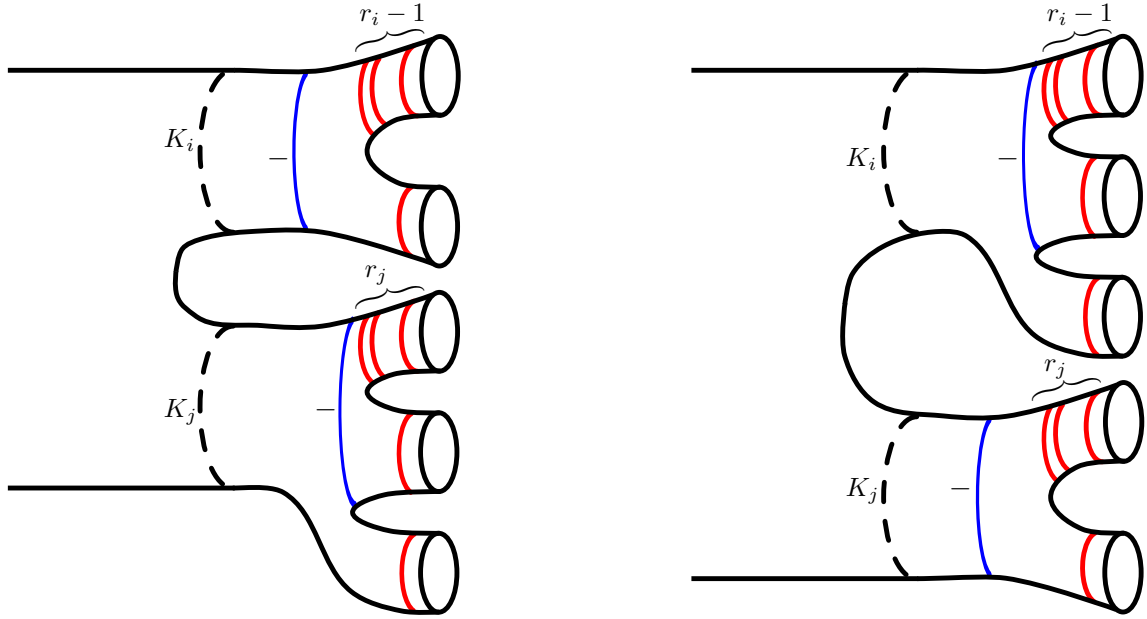
\begin{figure}[htbp]
\begin{center}
\begin{overpic}[scale=2.3,tics=30]{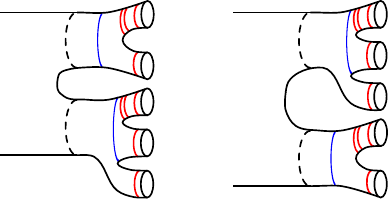}
\put(130,205){\small \rotatebox{105}{$\left.\begin{array}{c} \vspace{.1cm} \\ ~ \end{array}\right\}$}}
\put(133,225){$r_i - 1$}
\put(390,205){\small \rotatebox{105}{$\left.\begin{array}{c} \vspace{.1cm} \\ ~ \end{array}\right\}$}}
\put(393,225){$r_i-1$}
\put(130,108){\small \rotatebox{105}{$\left.\begin{array}{c} \vspace{.1cm} \\ ~ \end{array}\right\}$}}
\put(135,128){$r_j$}
\put(386,73){\small \rotatebox{105}{$\left.\begin{array}{c} \vspace{.1cm} \\ ~ \end{array}\right\}$}}
\put(392,92.5){$r_j$}
\put(356,45){\large $-$}
\put(116,75){\large $-$}
\put(98,170){\large $-$}
\put(374,170){\large $-$}
\put(59,177){$K_i$}
\put(59,75){$K_j$}
\put(317,170){$K_i$}
\put(317,45){$K_j$}
\end{overpic}
\caption{On the left, we have performed $r_i$-surgery on $K_i$, and $r_j$-surgery on $K_j$ with respect to the page slope, although each is expressed in different ways.  After moving the bottom boundary component in the left picture to be in the middle in the right picture, we see that open book is the same as that for $(r_i-1)$-surgery on $K_i$ and $(r_j+1)$-surgery on $K_j$.  In these open books, the red unmarked curves represent positive Dehn twists and the blue curves marked with a minus sign represent negative Dehn twists.}
\label{fig:link surgery}
\end{center}
\end{figure}

\subsection{Contact Width and Other Invariants}

The \textit{contact width} $w(K)$ of a topological knot $K$ in a contact manifolds $\Mxi$ is the supremum of the slopes of the dividing curves of the convex boundaries of regular neighbourhoods of transverse knots in the knot type $K$.  This has been in general tough to calculate, and has been worked out only in very specific cases (for example, torus knots in $(S^3,\xi_{\rm{std}})$).  Our results allow us to easily calculate the contact width of a large class of knots.

\begin{proof}[Proof of Corollary~\ref{cor:contact width}]
Since there is a Legendrian approximation $L$ of $K$ with $tb(L) = 2g-1$, and a regular neighbourhood of $L$ with convex boundary with dividing curve slope $2g-1$ is a regular neighbourhood of $K$, we know that $w(K) \geq 2g-1$.  If $w(K) > 2g-1$, then $K$ would have a regular neighbourhood $S$ with convex boundary with dividing curves of slope $s > 2g-1$.  Theorem~\ref{thm:preserve tightness} implies that inadmissible transverse $r$-surgery on $S$ is tight for all $r > 2g-1$.  In particular, inadmissible transverse $s$-surgery on $K$ is tight. But the Legendrian divides on $\bd T$ now bound an overtwisted disc in $M_s(K)$, which is a contradiction.
\end{proof}

\begin{remark}
For knots in $(S^3,\xi_{\rm{std}})$, we could use results of Lisca and Stipsicz \cite{LS:hf1} instead of Theorem~\ref{thm:preserve tightness} to prove Corollary~\ref{cor:contact width}.
\end{remark}

\begin{remark}
The knots that are closures of positive braids in $(S^3,\xi_{\rm{std}})$ are fibred and have Legendrian representatives with $tb = 2g-1$, see \cite{Kalman}. Corollary~\ref{cor:contact width} then implies that the contact width of all these knots is $2g-1$.  Other examples outside of $S^3$ are plentiful, but it is unclear what large families of knots are interesting and easily describable.
\end{remark}

In another direction, Baldwin and Etnyre \cite{BE:transverse} have defined an invariant $t(K)$ of a transverse knot $K$ in a contact manifold $\Mxi$.  Given a standard neighbourhood $N$ of $K$ with boundary slope $a$, we extend their definition to include inadmissible transverse surgery and define
\begin{align*}
t(K,N) = \{ r \in \Q \,|\,&\mbox{admissible (resp.\ inadmissible) transverse $r$-surgery} \\ 
&\mbox{on $K$ using $N$ is tight, where $r < a$ (resp.\ $r \geq a$)}\}.
\end{align*}
When either the result is independent of the neighbourhood $N$, or $N$ is understood, we leave it off, and just right $t(K)$.

From Lisca and Stipsicz \cite{LS:hf1}, we know examples in $S^3$ where $t(K,N) = \R\backslash\overline{tb}(L)$; for example, knots who have Legendrian approximations with $tb(L) = 2g(L)-1$, like the positive torus knots, and $N$ is a neighbourhood of a maximum $tb$ representative of $K$.  The inadmissible half of $t(K,N)$ follows from \cite{LS:hf1}, and the admissible half follows from work of Etnyre, LaFountain, and Tosun \cite{ELT:cables} in the cases where the surgery is not representable by negative contact surgery. Given Theorem~\ref{thm:all OT surgeries}, we now have examples inside $S^3$ where $t(K,N)$ is not almost all of $\R$; for example, with Remark~\ref{overtwistedpositivesurgeries} for negative torus knots, we see that $t(K) = [-\infty,\overline{tb}(K))$ (the negative torus knots are uniformly thick, see \cite{EH:cables}, and so the result is independent of $N$).  Thanks to Corollary~\ref{cor:contact width}, we have more information on this invariant for knots outside of $S^3$.

\section{Fillability and Universal Tightness}
\label{sec:fillability}

Having seen that the model open books constructed in Section~\ref{sec:tight surgeries} are tight, and have non-vanishing Heegaard Floer contact invariant, we now ask about their other properties, namely, fillability and universal tightness.

\subsection{Fillability}
\label{subsec:fillability}

A contact 3-manifold $\Mxi$ is \textit{weakly fillable} if there exists a symplectic 4-manifold $(X, \omega)$ with a compatible almost-complex structure $J$ such that $\Mxi$ is orientation-preserving contactomorphic to $(\partial X,T\partial X \cap J(T\partial X))$, and $\omega$ restricted to $T\partial X \cap J(T\partial X)$ is positive.  We say that it is \textit{weakly semi-fillable} if it is the weak convex boundary of a 4-manifold that may have other convex boundary components.  We say that $\Mxi$ is \textit{strongly fillable} if there is a weak filling $(X, J, \omega)$, where $(X, \omega)$ is a symplectic 4-manifold, and there is a vector-field $v$ pointing transversely out of $X$ along $M$, such that $\xi = \ker \iota_v\omega$, and the flow of $v$ preserves $\omega$.  It is \textit{Stein fillable} if $(X,J,\omega)$ is \textit{Stein}.

We know from Theorem~\ref{thm:monoid} that there are monoids in mapping class group for monodromies corresponding to the different fillability conditions listed above.  If the model open books we constructed were weak (resp.\ strongly, Stein) fillable, then large enough positive surgeries on fibred knots supporting weakly (resp.\ strongly, Stein) fillable contact structures would be weakly (resp.\ strongly, Stein) fillable as well.  We show that these open books are not weakly fillable, and hence not strongly or Stein fillable.  Since these open books correspond to positive surgeries on the standard tight contact structures on connect sums of copies of $S^1 \times S^2$, which are Stein fillable, we see that no fillability condition on the original contact structure is enough by itself to guarantee fillability of the surgered manifold.

\begin{figure}[htbp]
\begin{center}
\begin{overpic}[scale=1.5]{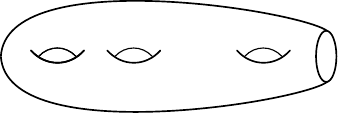}
\put(134,38){\LARGE $\cdots$}
\end{overpic}
\caption{The open book $(\Sigma_g^1,id_{\Sigma_g^1})$ with trivial monodromy.}
\label{fig:genus g trivial}
\end{center}
\end{figure}

We begin by noting that the binding $K$ of the genus $g$ open book with connected binding and trivial monodromy $(\Sigma_g^1,id_{\Sigma_g^1})$ is a transverse knot in the contact structure supported by the open book with $sl(K) = 2g-1$ (see Figure~\ref{fig:genus g trivial}).  By the discussion in the proof of Lemma~\ref{lem:model OBD HF}, inadmissible transverse surgery on $K$ is contactomorphic to positive contact surgery on the Legendrian knot $L$ in Figure~\ref{fig:legendrian}.  We will show that contact $(n-2g+1)$-surgery on $L$ is not weakly fillable for any odd integer $n \geq 2g$, and since fillability is preserved by contactomorphism, the same is true of inadmissible transverse $n$-surgery on $K$.  This will be enough for us to conclude that inadmissible transverse $r$-surgery on $K$ is not weakly fillable for any rational $r> 0$.

Since $L$ has $tb(L) = 2g-1$ and $rot(L) = 0$, inadmissible transverse $n$-surgery on $K$ is contactomorphic to contact $(n-2g+1)$-surgery on $L$, for $n \geq 2g$.  Using Construction~\ref{DingGeiges}, we take $L_1$, a negatively stabilised push-off of $L$, and $L_2, \ldots, L_{n-2g}$, push-offs of $L_1$, and perform contact $(+1)$-surgery on $L$ and contact $(-1)$-surgery on $L_i$, $i = 1, \ldots, n-2g$.

We will calculate the homotopy invariants of the contact structure obtained by this surgery.  The formulae used here are from Ding, Geiges, and Stipsicz \cite{DGS:surgery}, based on work of Gompf \cite{Gompf}.

\begin{prop}\label{prop:homotopy for fillable}
If $\Mxi$ is the contact manifold resulting from contact $(n-2g+1)$-surgery on $L$ (with all negative stabilisation choices), for $n \geq 2g$, then the homotopy invariants of $\xi$ are $\PD c_1(\xi) = (n-2g)[\mu_L]$ and $d_3(\xi)=\frac{4g^2-3n+n^2}{4n}$.
\end{prop}
\begin{proof}
Let $\Mxi$ be the resulting manifold.  The Poincaré dual of $c_1(\xi)$ is given by
$$\PD c_1(\xi) = rot(L) [\mu_L] + \sum_{i=1}^{n-2g} rot(L_i) [\mu_{L_i}] = -\sum_{i=1}^{n-2g} [\mu_{L_i}].$$
Since the topological surgery diagram is given by Figure~\ref{fig:n surgery kirby diagram}, we see that $[\mu_{L_i}] = -[\mu_L]$ in $H_1(M)$, and so $$\PD c_1(\xi) = \left(n-2g\right)[\mu_L].$$

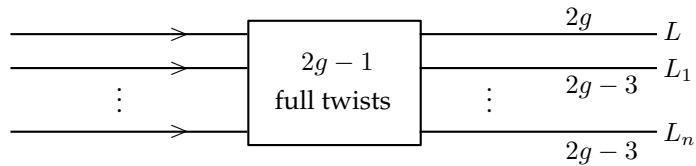
\begin{figure}[htbp]
\begin{center}
\begin{overpic}[scale=1,tics=10]{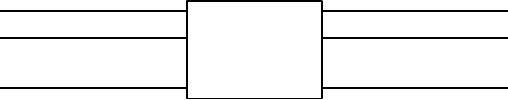}
\put(110,29){$2g-1$}
\put(100,14){full twists}
\put(40,13){$\vdots$}
\put(180,13){$\vdots$}
\put(247,40){$L$}
\put(247,26){$L_1$}
\put(247,2){$L_n$}
\put(210,46){$2g$}
\put(210,20){$2g-3$}
\put(210,-5){$2g-3$}
\put(60,39.93){$>$}
\put(60,27.05){$>$}
\put(60,3){$>$}
\end{overpic}
\caption{Topological surgery description of $M$.}
\label{fig:preslide n surgery kirby diagram}
\end{center}
\end{figure}

To work out the three-dimensional invariant $d_3(\xi)$, we consider the almost-complex manifold $(X,J)$ with boundary $\Mxi$ obtained by taking $B^3$, attaching $2g$, $1$-handles (corresponding to the connect sum of $S^1 \times S^2$ supported by the open book in Figure~\ref{fig:genus g trivial}), and then attaching $n-2g+1$, $2$-handles along $L, L_1, \ldots, L_{n-2g}$ with framings $2g, 2g-3, \ldots, 2g-3$, see Figure~\ref{fig:preslide n surgery kirby diagram}.

From the construction, we see that $\chi(X) = 1 - 2g + (n-2g+1) = 2-4g+n$.  The linking matrix for the 2-handle surgeries is given by
$$N = \left(\begin{array}{cccccc}2g&2g-1&2g-1&\cdots&2g-1&2g-1 \\2g-1&2g-3&2g-2&\cdots&2g-2&2g-2 \\2g-1&2g-2&2g-3&\cdots&2g-2&2g-2 \\ \vdots & \vdots & \vdots & \ddots & \vdots & \vdots \\ 2g-1&2g-2&2g-2&\cdots&2g-3&2g-2 \\ 2g-1&2g-2&2g-2&\cdots&2g-2&2g-3 \end{array}\right).$$
We claim that this matrix has signature $1-(n-2g)$, and thus that $\sigma(X) = 1-n+2g$.  Indeed, by handle-sliding $L_i$ over $L$, for $i = 1, \ldots, n-2g$, we get the topological surgery picture in Figure~\ref{fig:n surgery kirby diagram}.   After handlesliding $L$ over  the $-1$ framed unknots, we get $n-2g$ unlinked copies of a $-1$ framed unknot, and $n$ surgery on $L$.  Note that this calculation is valid even in the presence of $1$-handles, as all knots involved in the calculation are null-homologous.

\begin{figure}[htbp]
\begin{center}
\begin{overpic}[scale=1.5,tics=15]{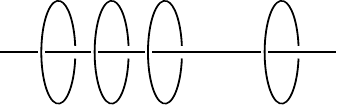}
\put(-7,30){$L$}
\put(5,42){$2g$}
\put(19,60){$L_1$}
\put(40,76){$-1$}
\put(58,60){$L_2$}
\put(80,76){$-1$}
\put(96,60){$L_3$}
\put(118,76){$-1$}
\put(150,60){\LARGE $\cdots$}
\put(220,60){$L_{n-2g}$}
\put(205,76){$-1$}
\put(155,35.3){$>$}
\put(28,50){\rotatebox{80}{$>$}}
\put(66.6,50){\rotatebox{80}{$>$}}
\put(105.1,50){\rotatebox{80}{$>$}}
\put(189.45,50){\rotatebox{80}{$>$}}
\end{overpic}
\caption{Topological surgery description of $M$ after handle-sliding $L_i$ over $L$, for $i = 1, \ldots, n-2g$.}
\label{fig:n surgery kirby diagram}
\end{center}
\end{figure}
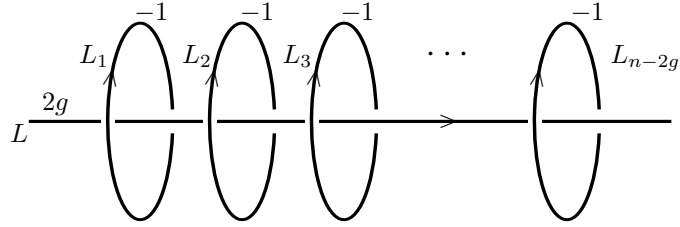

To calculate $c_1^2(X,J)$, we first note that $n\cdot\PD c_1(\xi) = 0$, and thus $n\cdot\PD c_1(\xi)$ is in the image of $N$.  We can check that $$N^{-1}\left(n \cdot \PD c_1(\xi)\right) = \left(\begin{array}{c}(n-2g)(1-2g)\\2g\\\vdots\\2g\end{array}\right).$$  Thus we calculate
$$c_1^2(X,J) = \frac{1}{n^2}\left(\begin{array}{c}(n-2g)(1-2g)\\2g\\\vdots\\2g\end{array}\right)\cdot \left(\begin{array}{c}0 \\ -n \\ \vdots \\ -n \end{array}\right) = \frac{4g^2-2gn}{n}.$$

We combine these results to calculate
$$d_3(\xi) = \frac{1}{4}\left(c_1^2(X,J) - 3\sigma(X) - 2\chi(X)\right) + 1 = \frac{4g^2-3n+n^2}{4n},$$
where the plus 1 comes from the contact $(+1)$-surgery on $L$.
\end{proof}

\begin{thm}\label{thm:non-fillable surgeries}
If $(M_r,\xi_r)$ is the contact structure coming from inadmissible transverse $r$-surgery on $K$, then $\xi_r$ is not weakly fillable for any rational $r > 0$.  In particular, inadmissible transverse surgery does not preserve fillability.
\end{thm}
\begin{proof}
Let $(M'_n,\xi'_n)$ be the contact manifold coming from contact $(n-2g+1)$-surgery on $L$, as in Proposition~\ref{prop:homotopy for fillable}. The values of $c_1(\xi'_n)$ and $d_3(\xi'_n)$ calculated in Proposition~\ref{prop:homotopy for fillable} are the same as those calculated by Lisca and Stipsicz \cite{LS:non-fillable} for Honda's examples \cite{Honda:classification2} of circle bundles over $\Sigma_g$ with Euler number $n$.  For $n$ odd, we can conclude that our examples are in the same $\Spinc$ structure as their $\xi_0$ (whereas for $n$ even, there is a 2:1 correspondence of $c_1$ to $\Spinc$ structures, and so we cannot conclude anything).  Lisca and Stipsicz \cite{LS:non-fillable} proved that a contact structure in this $\Spinc$ structure and with this $d_3$ invariant cannot be weakly fillable, or even weakly semi-fillable.

By \cite{BE:transverse}, $\xi'_n$ is achievable by negative contact surgery on a link in $\xi_r$ for any rational $0 < r < n$, and negative contact surgery preserves weak fillability \cite{EH:non-fillable}.  Thus, the fact that $\xi'_n$ is not weakly fillable for any odd integer $n$ implies that $\xi_r$ is not weakly fillable for any rational $r > 0$.
\end{proof}

\begin{remark} We calculated $c_1(\xi)$ using a Legendrian surgery diagram.  It is possible to calculate this directly from the transverse surgery picture.  See \cite{conwaythesis} for details. \end{remark}

\subsection{Universal Tightness}

Recall that given a cover $M' \to M$ and a contact structure $\xi$ on $M$, there is an induced contact structure $\xi'$ on $M'$.  We say $\Mxi$ is \textit{universally tight} if $\xi$ is tight and the induced contact structure on every cover is also tight; otherwise, $\Mxi$ is \textit{virtually overtwisted}.  Since the fundamental group of a compact $3$-manifold is residually finite (by geometrization), universal tightness is equivalent to requiring that all finite covers of $\Mxi$ remain tight \cite{Honda:classification1}.

\begin{remark} Honda has classified \cite{Honda:classification2} tight contact structures on $S^1$ bundles over $T^2$ with Euler class $n$, and shown that there are exactly one (resp.\ two) virtually overtwisted contact structures when $n = 2$ (resp.\ $n > 2$).  Van Horn-Morris \cite{VHM} has shown that all the universally tight contact structures that Honda classifies are weakly fillable.  Thus, by Theorem~\ref{thm:non-fillable surgeries}, we can conclude that the contact manifolds resulting from inadmissible transverse $n$-surgery on the binding of $(\Sigma^1_1, id_{\Sigma_g^1})$ support virtually overtwisted contact structures for any integer $n \geq 2$. \end{remark}

\begin{thm}
The result of inadmissible transverse $r$-surgery on the binding of $(\Sigma^1_g, \rm{id}_{\Sigma^1_g})$ is virtually overtwisted for any $r > 0$.  In particular, inadmissible transverse surgery does not preserve universal tightness.
\end{thm}
\begin{proof}
Let $\Sigphi$ be an open book for $\Mxi$ with binding $L$.  Given a finite cover $\pi:M' \to M$, the link $\pi^{-1}(L) \subset M'$ is fibred, and induces an open book with page $\Sigma' = \pi^{-1}(\Sigma)$ with induced monodromy $\phi'$ on $\Sigma'$.  We will show that the open books for inadmissible transverse $r$-surgery on the binding $K$ of the open book $(\Sigma_g^1,\rm{id}_{\Sigma_g^1})$ in Figure~\ref{fig:genus g trivial} are virtually overtwisted for any $r > 0$.

First note that for $g > 0$, the surface $\Sigma=\Sigma_g^1$ admits a double cover by $\Sigma'=\Sigma_{2g-1}^2$.  Let $K$ be the binding component in $(\Sigma,\rm{id}_{\Sigma})$, and let $K_1$ and $K_2$ be the induced binding components in $(\Sigma',\rm{id}_{\Sigma'})$.  It is not hard to see that taking the double cover of the open book resulting from inadmissible transverse $r$-surgery on $K$ is the same as doing inadmissible transverse $r$-surgery on each of $K_1$ and $K_2$ with respect to the framings induced by $\Sigma'$.

By the proof of Theorem~\ref{thm:link surgery}, the contact manifold resulting from inadmissible transverse $r$-surgery on each of $K_1$ and $K_2$ is the same as that resulting from inadmissible transverse $(r-n)$-surgery on $K_1$ and inadmissible transverse $(r+n)$-surgery on $K_2$, for any integer $n$.  In particular, if $n > r$, then we can explicitly find an overtwisted disc in a neighbourhood of the dual knot to $K_1$ in the surgered manifold, and thus the result is overtwisted.

Since this contact manifold is a double cover of the contact manifold obtained by inadmissible transverse $r$-surgery on $K$, we conclude that the latter contact manifold is virtually overtwisted.
\end{proof}

\section{Overtwisted Surgeries}
\label{sec:OT surgeries}

In this section, we prove Theorem~\ref{overtwistedsurgery} and its generalisations.  We have the following classification, due to Mark and Tosun, of when inadmissible transverse surgery on a transverse knot $T$ in $(S^3,\xi_{\rm{std}})$ has non-vanishing contact Heegaard Floer invariant.

\begin{thm}[Mark and Tosun \cite{MT}]\label{transversemt} Let $T \subset (S^3, \xi_{\rm{std}})$ be a transverse knot of smooth knot type $K$.  Then the Heegaard Floer contact invariant of inadmissible transverse $r$-surgery on $T$ is non-vanishing for a rational number $r$ if and only if $sl(K) = 2\tau(K) - 1$ and
\begin{itemize}
\item if $\epsilon(K) = 1$, then $r > 2\tau(K) - 1$,
\item if $\epsilon(K) = 0$, then $r \geq 2\tau(K)$.
\end{itemize}
\end{thm}

Here, $\epsilon(K)$ is a knot invariant coming from Heegaard Floer homology that takes values in $\{0, \pm 1\}$.  Mark and Tosun's original paper was in the context of positive contact surgeries on Legendrian knots, and can be converted back to the original formulation by replacing $sl$ with $tb - rot$ and $r$ with $r - tb$.  There is currently no example of a transverse knot $(S^3,\xi_{\rm{std}})$ where some inadmissible transverse surgery is tight, yet $T$ does not satisfy the conditions of Theorem~\ref{transversemt}.  The natural question is whether those surgeries whose contact invariant vanishes are indeed overtwisted.  Theorem~\ref{overtwistedsurgery} gives examples where the surgery is overtwisted, lending support to a positive answer to the following question.

\begin{question} Is the tightness of surgeries on transverse knots in $S^3$ characterised by the contact class of the surgered manifold? \end{question}

Given a Legendrian $L\subset \Mxi$, consider the core of the surgery torus $L^*$ in $(M_{tb(L)+1}(L),\xi_{(+1)}(L))$.  This is naturally a Legendrian knot, and its contact framing agrees with the contact framing of $L$ in $\Mxi$.  A single negative (resp.\ positive) stabilisation of $L^*$ gives us $L^*_-$ (resp.\ $L^*_+$).  To work out the contact framing of $L^*_{\pm}$ measured by the meridian and contact framing of $L$, we pass to the Farey graph.  Indeed, although the Farey graph was described in Section~\ref{sec:construction} with reference to transverse knots in the binding of open books, the same operations can be used to describe the meridional slope and contact framing of Legendrian knots: a negative stabilisation of $L$ is operation (1), while contact $(+1)$-surgery (resp.\ $(-1)$-surgery) on $L$ correspond to operation (2) (resp.\ (3)), \textit{cf.\@} Proposition~\ref{prop:plusminus one surgery}.

For $L$, we set $\infty$ to be the meridian slope (West's label), and $0$ to be the contact framing (East's label).  Thus contact $(+1)$-surgery moves the label $1$ to West, and keeps East's label fixed.  The labels of the points that in the standard labeling of the Farey graph that are labeled with integers (East, North, North-West, West-North-West, $\ldots$, South, South-West, West-South-West, $\ldots$) are possible contact framings for the knot type of $L^*$.  Since the contact framing of $L^*$ is $0$, and that label is at East, a stabilisation will take the next largest contact framing (with respect to the new labeling), which is at North.  This is labeled $\infty$, and hence it is the slope that corresponds to the meridian of $L$.

Thus, the complement of a standard neighbourhood of $L^*_-$ (resp.\ $L^*_+$) with convex boundary is the manifold obtained from the complement of $L$ by attaching a negative (resp.\ positive) bypass to get to the meridian slope $\infty$ as the slope of the dividing curves.  We claim that if these sutured manifolds are overtwisted, then all positive contact surgeries on $L$ are overtwisted.  This is because when performing positive contact surgery, the first step is to add a bypass layer to get to the meridian slope.  This manifold embeds into the surgered manifold, thus if it is overtwisted, then so is the surgered manifold.  If we can only show that the complement of a standard neighbourhood $L^*_-$ is overtwisted, then we can show that all inadmissible transverse surgeries are overtwisted (corresponding to negative bypass layers), but not all positive contact surgeries in general.

To show that these manifolds with boundary are overtwisted, we use tools previously used to detect \textit{loose} Legendrian knots, that is, Legendrian knots where the complement of a standard neighbourhood is overtwisted.  To do this, we follow the lead of Baker and Onaran \cite{BO:non-loose}, using the calculation tools of Lisca, Ozsváth, Stipsicz, and Szábo \cite{LOSS}.

\subsection{Invariants of Rationally Null-Homologous Legendrian Knots}

In this section, we define the rational Thurston--Bennequin $tb_\Q(L)$ and the rational rotation number $rot_\Q(L)$ of a rationally null-homologous Legendrian knot $L$.  See \cite{BE:rational} for more details and properties.

Given a rationally null-homologous Legendrian knot $L\subset \Mxi$, where $M$ is a rational homology sphere, let $\Sigma$ be a \textit{rational Seifert surface} for $L$ with connected binding.  That is, $\bd\Sigma$ is connected and is homologous to $r\cdot[L]$, where $r$ is the smallest positive integer such that $r \cdot [L] = 0 \in H_1(M;\Z)$.  Given another Legendrian knot $L'$, we define the \textit{rational linking} to be $$lk_\Q(L,L') = \frac{1}{r}\,[\Sigma]\cdot [L'].$$
Consider the framing of the normal bundle of $L$ induced by $\xi|_L$.  Let the Legendrian knot $L'$ be a push-off of $L$ in the direction of this framing.  Then we define the \textit{rational Thurston--Bennequin} number of $L$ to be $$tb_\Q(L) = lk_\Q(L,L').$$
Let $\iota: \Sigma \to M$ be an embedding on the interior of $\Sigma$.  We choose a trivialisation $\tau$ of the pull-back bundle $\iota^*(\xi)$ over $\Sigma$.  Along $\bd\Sigma$, $\tau$ gives an isomorphism of the bundle to $\bd\Sigma \times \R^2$.The tangent vector to $\bd\iota(\Sigma)$ gives a framing of $\xi|_L$, so its pullback $v$ gives a framing of $\iota^*(\xi)$ along $\bd\Sigma$.  We define the \textit{rational rotation number} of $L$ to be $$rot_\Q(L) = \frac{1}{r}\,wind_\tau(v),$$ where $wind_\tau(v)$ measures the winding number of $v$ in $\R^2$ with respect to the trivialisation $\tau$.

A Legendrian knot $L$ in an overtwisted contact manifold is called \textit{loose} if the complement of a standard neighbourhood of $L$ is overtwisted; otherwise, it is called \textit{non-loose}.

\begin{lem}[Świątkowski \cite{Dymara}, Etnyre \cite{Etnyre:OT}, Baker--Onaran \cite{BO:non-loose}]\label{lem:S bound} If $L \subset \Mxi$ is a rationally null-homologous Legendrian knot such that the complement of a regular neighbourhood of $L$ is tight, then $$-|tb_\Q(L)| + |rot_\Q(L)| \leq -\frac{\chi(L)}{r},$$ where $r$ is the order of $[L]$ in $H_1(M; \Z)$. \end{lem}

In fact, Baker and Etnyre showed \cite{BE:rational} that for a rationally null-homologous Legendrian knot in a tight contact manifold, the rational homotopy invariants satisfy $tb_\Q(L) + |rot_\Q(L)| \leq -\chi(L)/r$.  This gives in some cases a better inequality than that from Lemma~\ref{lem:S bound}, but this does not improve the results of this section.

The following lemma has been proved by Lisca, Ozsváth, Stipsicz, and Szábo \cite{LOSS} and Geiges and Onaran \cite{GO:rationalunknots} for surgeries in $(S^3,\xi_{\rm{std}})$.  We extend it to surgeries in a more general contact manifold.

\begin{lem}\label{lem:GO} Let $L_0 \cup \cdots \cup L_n \subset \Mxi$ be a collection of null-homologous Legendrian knots, where $c_1(\xi)$ is torsion.  Perform contact $(\pm 1)$-surgery on $L_i$, for $i = 1, \ldots, n$ (the sign need not be the same for each $i$), and let $a_i$ be the topological surgery coefficient.  Assume that the resulting manifold $(M',\xi')$ has the same rational homology as $M$.  Let $N = (N_{ij})$ for $1 \leq i,j\leq n$ be the matrix given by $$N_{ij} = \left\{ \begin{array}{lr} a_i \,\,\,\,\,& i = j, \\ lk(L_i,L_j) \,\,\,\,\,& i \neq j, \end{array}\right.$$ and let $N_0 = ((N_0)_{ij})$ for $0 \leq i,j \leq n$ be the matrix given by
$$(N_0)_{ij} = \left\{ \begin{array}{lr} 0 \,\,\,\,\,&i = j = 0, \\ a_i \,\,\,\,\,& i = j \geq 1, \\ lk(L_i,L_j) \,\,\,\,\,& i \neq j. \end{array}\right.$$
Then the rational classical invariants for $L$, the image of $L_0$ in $(M',\xi')$, are
$$tb_\Q(L) = tb(L_0) + \frac{\det N_0}{\det N},$$ and $$rot_\Q(L) = rot(L_0) - \left<\left(\begin{matrix} rot(L_1) \\ \vdots \\ rot(L_n) \end{matrix}\right), N^{-1}\left(\begin{matrix} lk(L_0,L_1) \\ \vdots \\ lk(L_0,L_n) \end{matrix}\right)\right>.$$
\end{lem}
\begin{proof}
We give a sketch of the argument, paying attention to where the details differ from \cite[Lemma 2]{GO:rationalunknots}.  For each $i = 0,\ldots,n$, let $\lambda_i$ and $\mu_i$ be the Seifert framing and meridian respectively for $L_i$ in $M$.  Because each $L_i$ is null-homologous, we can conclude that $$H_1(M'\backslash L) \cong H_1(M) \oplus \Big(\Z\langle\mu_0\rangle \oplus \cdots \oplus \Z\langle\mu_n\rangle\Big) / \langle a_i\mu_i + \sum_{\substack{j=0 \\ j\neq i}}^n lk(L_i,L_j)\mu_j = 0,i=1,\ldots,n\rangle.$$

From the Mayer-Vietoris sequence of $M' = M' \backslash L \cup L$, we get the short exact sequence $$0 \to \Z\langle \mu_0\rangle \oplus \Z\langle\lambda_0\rangle \to H_1(M'\backslash L) \oplus H_1(L) \to H_1(M') \to 0.$$
Note that $\Z\langle\lambda_0\rangle \to H_1(L)$ is an isomorphism, and $\mu_0$ maps to $0$ in $H_1(L)$.  Note also that the $H_1(M)$ summand in $H_1(M'\backslash L)$ maps isomorphically onto the $H_1(M)$ summand in $H_1(M')$, and the other summands of $H_1(M'\backslash L)$ map into the other summands of $H_1(M')$.  Thus we can get the short exact sequence $$0 \to \Z\langle\mu_0\rangle \to H_1(M'\backslash L)/H_1(M) \to H_1(M')/H_1(M) \to 0.$$
Since $H_1(M';\Q) = H_1(M;\Q)$, the preceding exact sequence considered with rational coefficients implies that the residue of $\PD c_1(\xi',L)$ in $H_1(M'\backslash L;\Q)/H_1(M;\Q)$ is some rational multiple of $\mu_0$.

To get a formula for $\PD c_1(\xi,\displaystyle\bigcup_{i=0}^n L_i)$, we start with a non-zero vector field $v$ over $L_i$.  Given Seifert surfaces $\Sigma_0, \ldots, \Sigma_n$ for $L_0, \ldots, L_n$ in $M$, we extend $v$ over $\Sigma_i$ such that there are $rot(L_i)$ zeroes over $\Sigma_i$.  Finally, we extend over the rest of $M$.  The zero set of $v$ tells us that $$\PD c_1(\xi,\bigcup_{i=0}^n L_i) = \sum_{i=0}^n rot(L_i)\mu_i + x,$$ where $x$ is the push-forward of some class in $H_1(M)$ that by construction does not intersect $\Sigma_i$.

We claim that we can construct a rational Seifert surface $\Sigma$ for $L$ in $M'$ such that $x$ acts trivially on $\Sigma$.  We then calculate that in $(M',\xi')$, if $L$ is order $r$ in $H_1(M')$, then
$$r\cdot rot_\Q(L) = \PD c_1(\xi',L) \cdot [\Sigma].$$  Notice that $\PD c_1(\xi',L)$ is the push-forward of $\PD c_1(\xi,\displaystyle\bigcup_{i=0}^nL_i)$, and $x$ acts trivially on $\Sigma$.  Thus with rational coefficients, the only free part left that could act non-trivially on $\Sigma$ is generated by $\mu_0$, and since $\mu_0 \cdot [\Sigma] = r$, it must be $rot_\Q(L)\mu_0$.

With rational coefficients, the summand of $\PD c_1(\xi,\displaystyle\bigcup_{i=0}^nL_i)$ corresponding to the $\mu_i$ can be written as an element of the $\Q$ summand of $H_1(M';\Q)$ generated by $\mu_0$. Thus we have the equation $$\sum_{i=0}^n rot(L_i)\mu_i = rot_\Q(L)\mu_0$$ in $H_1(M\backslash L_0;\Q)$.  Note that the surgery gives a cobordism $X:M \to M'$, where $H_2(X) = H_2(M) \oplus \Z^n$ and $H_2(X,M) = \Z^n$.  Thus the long-exact sequence of the pair $(X,M)$ gives $$H_2(M) \to H_2(X) \to H_2(X,M),$$ where the first map is an isomorphism into the $H_2(M)$ summand of $H_2(X)$, and the second map is $0$ on the $H_2(M)$ summand, and acts as the matrix $N$ on the $\Z^n$ summand.  Thus we see that $N$ is an injective map, and thus with rational coefficients, we can invert it.  The formula for $rot_\Q(L)$ then follows.

Let $\widetilde{\Sigma}_i$ be $\Sigma_i$ minus the interior of $N(L_0) \cup \cdots N(L_n)$.  To prove the claim, we construct a surface $\widetilde{\Sigma}$ in $M$ in a neighbourhood of $\widetilde{\Sigma}_0 \cup \cdots \cup \widetilde{\Sigma}_n$ such that its image in $M'$ can be capped off (by meridians of the surgery duals to $L_1, \ldots, L_n$) to a rational Seifert surface $\Sigma$ for $L$.  Since $x$ acts trivially on each $\Sigma_i$, it will also act trivially on $\widetilde{\Sigma}$.

First note that for every positive intersection of $L_i$ and $\Sigma_j$, the intersection of $\bd N(L_i)$ with $\widetilde{\Sigma}_j$ is $-\mu_i$, a meridian that in $M$ links $L_i$ once negatively, see FIgure~\ref{fig:meridian from linking}.  We start with $r$ copies of $\widetilde{\Sigma}_0$, and we would like to pick $|k_i|$ copies of $\widetilde{\Sigma}_i$, $i = 1, \ldots, n$, such that $$k_i\lambda_i - \left(r\cdot lk(L_0,L_i) + \sum_{\substack{j=1 \\ j\neq i}}^n k_j\cdot lk(L_i,L_j)\right)\mu_i = k_i \left(\lambda_i+a_i\mu_i\right).$$  If $k_i < 0$, we reverse the orientation of $\widetilde{\Sigma}_i$. This system of equations corresponds to the intersection of the collection of surfaces with $\bd N(L_i)$, for each $i$.  Comparing the coefficients of $\mu_i$ in each equation, we see that $$\sum_{i=1}^n k_i\cdot lk(L_i,L_j) = -r\cdot lk(L_0,L_i),$$ where we define $lk(L_i,L_i)$ to be $a_i$.  Notice that this is the same as the equation $$N\left(\begin{matrix} k_1 \\ \vdots \\ k_n \end{matrix}\right) = \left(\begin{matrix} -r\cdot lk(L_0,L_1) \\ \vdots \\ -r\cdot lk(L_0,L_n) \end{matrix}\right).$$  Since $N$ is invertible, we can solve for $k_i$.  We claim that each $k_i$ is an integer. To see this, first note that the order $r$ of $L$ is the order of $[L] = \displaystyle\sum_{i=1}^n lk(L_0, L_i)\mu_i$ in $H_1(M'; \Z)$.  Since $$-r\cdot\sum_{i=1}^n lk(L_0, L_i)\mu_i = 0$$ in $H_1(M'; \Z)$, we know that it is a sum of the relations $$a_i\mu_i + \sum_{\substack{j=0 \\ j\neq i}}^n lk(L_i,L_j)\mu_j.$$  Putting these quantities in vector form, this is equivalent to $$\left(\begin{matrix} -r\cdot lk(L_0,L_1) \\ \vdots \\ -r\cdot lk(L_0,L_n) \end{matrix}\right)$$ being an integer linear combination of the columns of $N$, where the coefficients of the linear combination are exactly $k_i$.

The boundary of the $2$-complex given by the union of $r$ copies of $\widetilde{\Sigma}_0$ and $k_i$ copies of $\widetilde{\Sigma}_i$, $i = 1, \ldots, n$ is homologous to an $(r,s)$ curve on $\bd N(L)$, for $s= \displaystyle\sum_{i=0}^n k_i \cdot lk(L_0,L_i)$, and $k_i$ copies of a $(1,a_i)$ curve on $\bd N(L_i)$.  Thus we can find some smooth embedded surface $\widetilde{\Sigma}$ in a neighbourhood of $\widetilde{\Sigma}_0 \cup \cdots \cup \widetilde{\Sigma}_n$ with boundary given by the oriented resolution of the boundary of the $2$-complex.  The boundary components of $\widetilde{\Sigma}$ on $\bd N(L_i)$, $i = 1, \ldots, n$, bound discs in $M'$, and capping off these components gives a rational Seifert surface $\Sigma$ for $L$ in $M'$.

Finally, we calculate $tb_\Q(L)$, by computing the intersection of the contact framing of $L$ in $(M',\xi')$ (which is the same as that of $L_0$ in $(M,\xi)$) and the rational Seifert slope, which is given by $a_0\mu_0 + r\lambda_0$, where $a_0$ is the unique integer such that the Seifert slope is null-homologous in $H_1(M'\backslash L)$.  The details are exactly the same as in \cite[Lemma 2]{GO:rationalunknots}.
\end{proof}

\begin{figure}[htbp]
\begin{center}
\begin{overpic}[scale=1,tics=20]{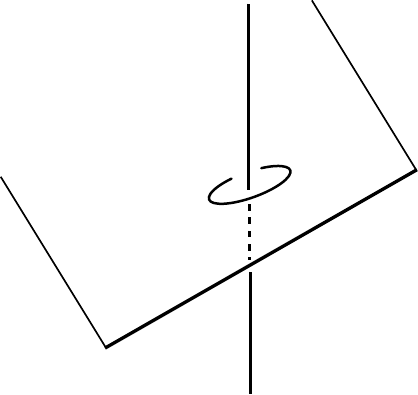}
\put(20,80){$\Sigma_j$}
\put(122,150){$L_i$}
\put(160,76){$L_j$}
\put(138,110){$-\mu_i$}
\put(117.3,130){\rotatebox{90}{$\textbf{>}$}}
\put(138,71){\rotatebox{34}{$\textbf{>}$}}
\put(125,94.3){\rotatebox{24}{$\textbf{<}$}}
\end{overpic}
\caption{When $L_i$ and $L_j$ link positively, the intersection of $\Sigma_j$ with the boundary of a neighbourhood of $L_i$ is $-\mu_i$.}
\label{fig:meridian from linking}
\end{center}
\end{figure}
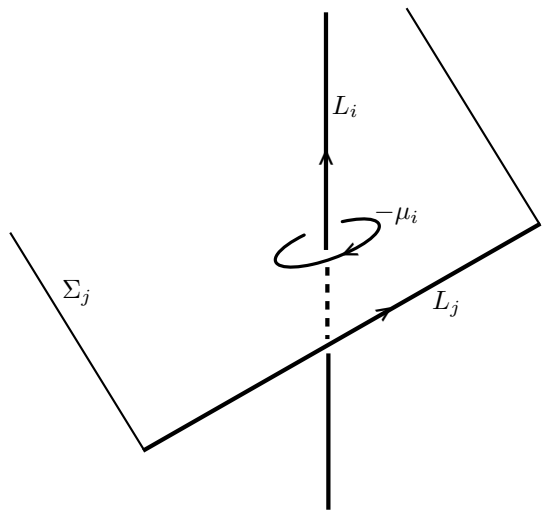

\begin{remark} If $c_1(\xi)$ is non-torsion, then rotation numbers may depend on the relative homology class of the Seifert surfaces that are chosen.  Thus, given Seifert surfaces $\Sigma_0,\ldots,\Sigma_n$, and using $rot(L_i,\Sigma_i)$ in the formulae, the same proof will calculate $rot_\Q(L,\Sigma)$, where $\Sigma$ is constructed from $\Sigma_0,\ldots,\Sigma_n$ as in the proof. For all our results in this section, the rational Seifert surface is constructed using a cabling process from the original Seifert surface.  However, for clarity, we state all our results in the context of $c_1(\xi)$ torsion. \end{remark}

\begin{remark}\label{remark:tb is topological} The proof of the formula for $tb_\Q(L)$ is entirely topological.  Thus, if we consider the contact surgery diagram as a smooth surgery diagram, and perform Kirby calculus moves on the diagram, then using the $M$ and $M_0$ from the new diagram will give the same value for $tb_\Q(L)$.  The calculations for $rot_\Q(L)$, however, are contact geometric in nature, and so must respect the contact surgery diagram chosen. \end{remark}

\subsection{Overtwisted Positive Contact Surgeries}

In this section, we prove Theorem~\ref{overtwistedsurgery} and its generalisations.  In particular, Theorem~\ref{overtwistedsurgery} is a combination of Theorem~\ref{thm:all OT surgeries} and the $n=1$ case of Theorem~\ref{thm:n surgery OT}.

The proof of the following theorem was inspired by (and runs similarly to) \cite[Theorem 4.1.8]{BO:non-loose}.

\begin{thm}\label{thm:all OT surgeries} If $K$ is a null-homologous transverse knot in $\Mxi$, $c_1(\xi)$ is torsion, and $sl(K) < -2g(K)-1$, then all inadmissible transverse surgeries on $K$ are overtwisted.  If $L$ is a Legendrian approximation of $K$, and in addition $rot(L) \leq 0$, then all positive contact surgeries on $L$ are overtwisted. \end{thm}
\begin{proof}
Let $L$ be a Legendrian approximation of $K$.  Consider contact $(+1)$-surgery on $L$, and let $L^*$ be the surgery dual.  This is topologically equivalent to $(tb(L)+1)$-surgery on $K$.  Consider a Legendrian push-off $L_0$ of $L$.  This is topologically a $(1, tb(L))$ curve on the boundary of a neighbourhood of $L$, where the longitude is given by the Seifert framing. In particular, $L_0$ is parallel to the dividing curves on the convex boundary of a standard neighbourhood of $L$.  Thus $L_0^*$, the image of $L_0$ after surgery on $L$, is still parallel to the dividing curves on the boundary of a standard neighbourhood of $L^*$, and thus is Legendrian isotopic to $L^*$.  Hence we conclude that $$\chi(L^*_0) = \chi(L^*) = \chi(L) = 1 - 2g(L).$$

We use Lemma~\ref{lem:GO} to work out the rational Thurston--Bennequin and rotation numbers of $L_0^*$.  We have $$N = (tb(L)+1)$$ and $$N_0 = \matrixb{0}{tb(L)}{tb(L)}{tb(L)+1}.$$  Thus since $tb(L_0) = tb(L)$, $rot(L_0) = rot(L)$, and $lk(L_0, L) = tb(L)$, we calculate
\begin{align*}
tb_{\Q}(L_0^*) &= tb(L_0) + \frac{\det N_0}{\det N} \\
&= tb(L) - \frac{tb(L)^2}{tb(L)+1} \\
&= \frac{tb(L)}{tb(L)+1}, \\
rot_{\Q}(L_0^*) &= rot(L_0) - rot(L)\cdot N^{-1}lk(L_0,L) \\
& = rot(L) - rot(L) \cdot \left(\frac{1}{tb(L)+1}\right)(tb(L)) \\
&= \frac{rot(L)}{tb(L)+1}.
\end{align*}

Consider $L^*_+$ and $L^*_-$, the positive and negative stabilisations of $L_0^*$.  As discussed above, the complement of $L^*_{\pm}$ is exactly the complement of $L$ in $N$ with a positive or negative basic slice added to the boundary to take the dividing curves of the boundary torus to meridional curves.  Thus, to show that the latter sutured manifolds are overtwisted, we show that $L^*_{\pm}$ are loose under the hypotheses of the theorem.  If adding a negative bypass to the complement of $L$ is overtwisted, then all inadmissible transverse surgeries on $K$ are overtwisted.  This is because the overtwisted sutured manifold embeds into all inadmissible transverse surgeries on $K$, as the surgeries correspond to choices of negative stabilisations in performing contact surgery (in particular, the first stabilisation is negative).  If adding the positive bypass and adding the negative bypass are both overtwisted, then any first choice of stabilisation in contact surgery will lead to an overtwisted sutured manifold, so all positive contact surgeries on $L$ are overtwisted.

We prove first that $L^*_-$ is loose by using Lemma~\ref{lem:S bound}.  We see that
\begin{align*}
tb_{\Q} (L_-^*) &= \frac{tb(L)}{tb(L)+1}-1 = \frac{-1}{tb(L)+1}, \\
rot_{\Q} (L_-^*) &= \frac{rot(L)}{tb(L)+1} - 1 = \frac{rot(L)-tb(L)-1}{tb(L)+1} = -\frac{sl(K) + 1}{tb(L)+1}.
\end{align*}

Plugging these into Lemma~\ref{lem:S bound}, our assumption that $sl(K) < -2g-1$ tells us that
$$-|tb_{\Q}(L_-^*)| + |rot_{\Q}(L_-^*)| = \frac{|sl(K)+1| - 1}{|tb(L)+1|} > \frac{2g-1}{|tb(L)+1|} = -\frac{\chi(L_-^*)}{|tb(L)+1|},$$ and so $L_-^*$ is loose.

Looking now at $L_+^*$, we calculate
\begin{align*}
tb_{\Q}(L_+^*) &= \frac{tb(L)}{tb(L)+1}-1 = \frac{-1}{tb(L)+1}, \\
rot_{\Q}(L_+^*) &= \frac{rot(L)}{tb(L)+1} + 1 = \frac{rot(L)+tb(L)+1}{tb(L)+1}
\end{align*}
so to break the bound of Lemma~\ref{lem:S bound}, we need $$|rot(L) + tb(L) + 1| > 2g.$$

The left-hand side of this inequality is equal to $|sl(K) + 1 + 2rot(L)|$.  Thus, if $rot(L) \leq 0$, then this inequality is satisfied, as $sl(K) < -2g -1$.  So both $L_-^*$ and $L_+^*$ are loose, as required.
\end{proof}

\begin{remark}\label{overtwistedpositivesurgeries} There are numerous classes of knots that fall in this category.  For example, a non-trivial negative torus knot $(-p, q)$ in $(S^3,\xi_{\rm{std}})$, where $p, q \geq 2$, has maximal self-linking $\overline{sl} = -pq$, and $2g - 1 = pq - p - q$.  Since $p, q \geq 2$, we see that $-pq < -pq + p + q -2=-(2g-1)-2=-2g - 1$.  Note that the overtwistedness of contact $(+1)$-surgery on these knots was previously known, by a result of Lisca and Stipsicz \cite{LS:hfnotes}.\end{remark}

\begin{remark} The figure-eight knot in $(S^3,\xi_{\rm{std}})$ has maximum self-linking $-3$, so it is not covered by Theorem~\ref{thm:all OT surgeries}.  However, it can be shown using convex surface theory that $L^*_-$ and $L^*_+$ are loose Legendrian knots, and thus that all positive contact surgeries on any figure-eight in $(S^3,\xi_{\rm{std}})$ are overtwisted.  See \cite{conway:f8} for details, as the convex surface theory involved in the proof is not in the spirit of the current paper.\end{remark}

In the following theorem, we look at contact $(+n)$-surgeries for some knots which are not covered by Theorem~\ref{thm:all OT surgeries}.

\begin{thm}\label{thm:n surgery OT}
Fix a positive integer $n \geq 1$.  If $L$ is a null-homologous Legendrian knot in $\Mxi$, $c_1(\xi)$ torsion, where $L$ is genus $g$, $tb(L) \leq -n-1$, and $|n\cdot rot(L)-(n-1)\cdot tb(L)| > n(2g-1) +tb(L)$, then contact $(+n)$-surgery on $L$ is overtwisted.
\end{thm}
\begin{remark}
For $n=1$, this inequality is simply $|rot(L)| > 2g-1+tb(L)$. Compare this result to that obtained by Lisca and Stipsicz \cite[Proposition 1.4]{LS:hfnotes}, which shows that contact $(+1)$-surgery on $L$ in $(S^3,\xi_{\rm{std}})$ has vanishing contact invariant when $tb(L) \leq -2$.  Note also that for fixed $g$, the region in the $(tb,rot)$-plane that is excluded by the hypotheses of this theorem contains finitely many points.
\end{remark}
\begin{proof}
Let $L$ be a Legendrian knot with $tb(L) = t$ and $rot(L) = r$.  We will deal with the negatively stabilised flavour of contact $(+n)$-surgery (note that the positively stabilised variety is equivalent to the negatively stabilised surgery on the reverse orientation of $L$).  Let $L_1$ and $L_2'$ be push-offs of $L$.  Stabilise $L_2'$ once negatively to get $L_2$, and let $L_3, \ldots, L_n$ be push-offs of $L_2$.  Contact $(+n)$-surgery on $L$ is equivalent to contact $(+1)$-surgery on $L_1$ and contact $(-1)$-surgery on $L_2, \ldots, L_n$.  Let $L^*$ be the image of $L$ after the surgeries on $L_1, \ldots, L_n$.  According to Remark~\ref{remark:tb is topological}, we only need a topological diagram of the surgery to calculate $tb_\Q(L^*)$.  Thus, we can assume we are doing topological $t+n$ surgery on a single knot $K$ that has linking $t$ with $L$.  Thus we can set $$N = (t+n) \mbox{ and } N_0 = \matrixb{0}{t}{t}{t+n}.$$  So we calculate that $$tb_\Q(L^*) = t + \frac{\det N_0}{\det N} = t - \frac{t^2}{t+n} = \frac{tn}{t+n}$$

In order to calculate $rot_\Q(L^*)$, however, we must use the $N$ that comes from the contact surgery diagram.  Thus we have the $n \times n$ matrix given by
$$N = \left(\begin{array}{cccccc}t+1&t&t&\cdots&t&t \\t&t-2&t-1&\cdots&t-1&t-1 \\t&t-1&t-2&\cdots&t-1&t-1 \\ \vdots & \vdots & \vdots & \ddots & \vdots & \vdots \\ t&t-1&t-1&\cdots&t-2&t-1 \\ t&t-1&t-1&\cdots&t-1&t-2 \end{array}\right).$$  It can be verified that its inverse is given by
$$N^{-1} = \frac{1}{t+n}\left(\begin{array}{cccccc}n-(n-1)t & t&t&\cdots&t&t \\ t & 1-n-t&1&\cdots&1&1 \\ t & 1 & 1-n-t & \cdots & 1 & 1 \\ \vdots&\vdots&\vdots&\ddots&\vdots&\vdots \\ t & 1 & 1 & \cdots & 1-n-t & 1 \\ t & 1 & 1 & \cdots & 1 & 1-n-t \end{array}\right).$$
We can then calculate that $$\left(\begin{array}{c}r\\ r-1\\ r-1\\ \vdots \\ r-1\end{array}\right)\cdot N^{-1}\left(\begin{array}{c}t\\ t\\ t\\ \vdots \\ t \end{array}\right) = \frac{1}{t+n}\left(\begin{array}{c}r\\ r-1\\ r-1\\ \vdots \\ r-1\end{array}\right)\cdot \left(\begin{array}{c}tn \\ -t \\ -t \\ \vdots \\ -t \end{array}\right) = \frac{(r+n-1)\cdot t}{t+n}.$$  Finally, we can conclude that
$$rot_\Q(L^*) = r - \frac{(n+r-1)\cdot t}{t+n} = \frac{rn-tn+t}{t+n}.$$

We now consider $k$ positive or negative stabilisations $L^*_{\pm k}$ of $L^*$, and plug the Thurston--Bennequin number and rotation number of $L^*_{\pm k}$ into Lemma~\ref{lem:S bound} to show that that the complement of $L^*_{\pm k}$ is overtwisted, and hence that the result of surgery on $L$ is overtwisted..  Let $L^*_k$ be the $k$-fold stabilisation of $L^*$ with stabilisation sign equal to the sign of $rot_\Q(L^*)$ (with any choice if the rotation vanishes).  Then for large $k$,
\begin{align*}
-|tb_\Q(L^*_k)| + |rot_\Q(L^*_k)| &= |rot_\Q(L^*)| + k-|tb_\Q(L^*) - k| \\
&= \frac{|rn-tn+t|}{|t+n|} + k - \left(k-\frac{|tn|}{|t+n|}\right) \\
&= \frac{|rn-tn+t| + |tn|}{|t+n|}.
\end{align*}
The first equality is true because the sign of the stabilisation is chosen to agree with the sign of $rot_\Q(L^*)$, and the second equality is true because $k$ is large and $tb_\Q(L^*)$ is positive.

We now need to work out the genus of $L^*_k$, \textit{ie.\@} the genus of $L^*$, in order to calculate $\chi(L^*)$.  We see that in $M$, $L$ is a $(1, t)$ cable of $K$, the single knot on which we performed topological $(t+n)$-surgery above to get the same manifold as contact surgery on $L_1, \ldots, L_n$ (an $(r,s)$-cable is a cable composed of $r$ longitudes and $s$ meridians).  Although topologically $L$ is isotopic to $K$, in general, the image $L^*$ of $L$ in $M_{t+n}(K)$ is not isotopic to $K^*$, the surgery dual knot to $K$; if $n=1$, then it is true that a push-off of a Legendrian knot gives a framing to the surgery dual to contact $(+1)$-surgery on the original Legendrian (and this is also true for contact $(-1)$-surgery), but this is false for general $n$. We claim that $L^*$ is in fact a $(-n,1)$-cable of $K^*$.  This can be seen by calculating the image of the cable under the map gluing the surgery torus into $M\backslash N(K)$, in the coordinate system where the longitude of $K^*$ is the meridian of $K$.
$$\matrixb{0}{1}{1}{t+n}^{-1}\vect{1}{t} = \matrixb{-t-n}{1}{1}{0}\vect{1}{t} = \vect{-n}{1}.$$ Letting $m = \gcd(n,|t+n|)$, we see that this knot has order $|t+n|/m$ in $H_1(M_{t+n}(K);\Z)$.  Note also, that the boundary of the Seifert surface traces a $\vects{-t-n}{1}=\vects{|t+n|}{1}$ curve on the boundary of a neighbourhood of $K^*$.  So we write
$$\frac{|t+n|}{m}\cdot\vect{-n}{1} = \frac{-n}{m}\cdot\vect{|t+n|}{1} - \frac{t}{m}\cdot\vect{0}{1},$$ where on the right, the first summand is copies of the Seifert surface, and the second summand is copies of the meridian of $K^*$.  Thus, a rational Seifert surface for $L^*$ is composed of $n/m$ copies of the rational Siefert surface $\Sigma$ for $K^*$ and $|t|/m$ copies of a meridional compressing disc for $K^*$, with bands corresponding to the intersections of $\vects{-n}{1}$ with the $|t|/m$ meridians.  So
$$\chi(L^*) = \frac{n}{m}\chi(\Sigma) + \frac{|t|}{m} - \frac{|nt|}{m} = \frac{n(1-2g)+nt-t}{m},$$ since $t < 0$ and $n > 0$.  Thus if $$\frac{|rn-tn+t|+|tn|}{|t+n|} > -\frac{\chi(L^*_k)}{|t+n|/m} = \frac{n(2g-1)+t-nt}{|t+n|},$$ then $L^*_k$ is loose and the contact structure on $M_{t+n}(K)$ is overtwisted.  Then our inequality is equivalent to requiring $$|rn-tn+t| > n(2g-1) +t.$$  Hence under our hypotheses, $L^*_k$ is loose as required.
\end{proof}

\begin{cor}
For every genus $g$ and every positive integer $n\geq 2$, there is a negative integer $t$ such that if $L$ is a null-homologous Legendrian knot of genus $g$ and $tb(L) \leq t$, then contact $(+n)$-surgery on $L$ is overtwisted.
\end{cor}

The following conjecture is motivated by the desire to remove the bounds on rotation number in the above theorems.

\begin{conj}\label{lots of overtwisted surgeries} If $L$ is a null-homologous Legendrian knot with $tb(L) \leq -2$, then contact $(+n)$-surgery on $L$ is overtwisted, for any positive integer $n < |tb(L)|$. \end{conj}

In contrast to the above results, we present an infinite family of Legendrian knots with arbitrarily low maximum Thurston--Bennequin number that admit tight positive contact surgeries.  These surgeries are topological $0$-surgeries, and so don't fit into the hypotheses of Conjecture~\ref{lots of overtwisted surgeries}.

\begin{prop}
For every positive integer $t$, there is an infinite class of null-homologous Legendrian knots in $(S^3,\xi_{\rm{std}})$ with $\overline{tb} = -t$, such that contact $(+t)$-surgery on each knot is tight.
\end{prop}
\begin{proof}
We use Theorem~\ref{transversemt} to prove that surgeries are tight.  We look for knots that are slice (\textit{ie.\@} $g_4 = 0$), as this implies that $\tau = \nu = 0$, and so the only requirement to have a positive contact surgery that is tight is that $\overline{sl} = 2\tau - 1 = -1$.  If $\overline{tb} = -t$, then Theorem~\ref{transversemt} says that contact $(+t)$-surgery is tight (and has non-vanishing contact invariant).

For $t = 1$, the slice knot $9_{46}$ has a Legendrian representative with $tb = -1$ and $rot = 0$.  For $t = 2$, the slice knot $8_{20}$ has a Legendrian representative with $tb = -2$ and $rot = -1$.  Now consider the knots $$K_{m,n} = \left(\mathop{\varhash}^{m} 8_{20}\right) \mathop{\varhash} \left(\mathop{\varhash}^n 9_{46}\right)$$ for non-negative integers $m$ and $n$, where the connect sum of 0 objects is taken to be the unknot.  Under connect sum, the slice genus is additive, so $K_{m,n}$ is slice for all $m, n$.  According to \cite{EH:knots2}, the rotation number of Legendrian knots under connect sum is additive, and the Thurston--Bennequin number adds like $$\overline{tb}(K_1 \mathop{\varhash} K_2) = \overline{tb}(K_1) + \overline{tb}(K_2) + 1.$$  Thus $$\overline{tb}(K_{m,n})  = -m-1$$ and there is a Legendrian representative with maximal Thurston--Bennequin and $$rot = -m.$$  Thus $\overline{sl}(K_{m,n}) = -1$, as required.
\end{proof}

\bibliography{references}{}
\bibliographystyle{plain}
\end{document}